\documentclass[12pt]{amsart}

\usepackage{tikz}
\usepackage{bbm}
\usetikzlibrary{3d,arrows,calc,positioning,decorations.pathreplacing,matrix} 

\usepackage{calligra,mathrsfs}
\usepackage[all]{xy}
\usepackage{float, comment}
\usepackage{mathtools}
\usepackage{amsmath}
\usepackage{amsthm}
\usepackage{amssymb}
\usepackage{amsbsy}
\usepackage{amstext} 
\usepackage{amsopn}
\usepackage[mathscr]{eucal}
\usepackage{enumerate}
\usepackage{xcolor}
\usepackage{graphicx} 
\usepackage{scalerel}
\usepackage{microtype} 
\usepackage{changepage} 
\usepackage[margin=1in,marginparwidth=0.8in, marginparsep=0.1in]{geometry}
\usepackage[pagebackref, bookmarks=true, bookmarksopen=true, bookmarksdepth=3,bookmarksopenlevel=2, colorlinks=true, linkcolor=blue, citecolor=blue, filecolor=blue, menucolor=blue, urlcolor=blue]{hyperref}

\usepackage{tikz}
\usepackage{tikz-cd}
\usepackage{bbm}
\usepackage[all]{xy}
\usetikzlibrary{arrows,calc,positioning,decorations.pathreplacing} 

\numberwithin{equation}{section}
\newtheorem{Theorem}[equation]{Theorem}
\newtheorem{Proposition}[equation]{Proposition} 
\newtheorem{Lemma}[equation]{Lemma}

\theoremstyle{definition}
\newtheorem{Remark}[equation]{Remark}

\numberwithin{figure}{section}

\newcommand{\bC}{\mathbb{C}}
\newcommand{\bD}{\mathbb{D}}

\newcommand{\bR}{\mathbb{R}}

\newcommand{\bZ}{\mathbb{Z}}

\newcommand{\cF}{\mathcal{F}}
\newcommand{\cG}{\mathcal{G}}

\newcommand{\cO}{\mathcal{O}}
\newcommand{\cP}{\mathcal{P}}

\newcommand{\fg}{\mathfrak{g}}

\newcommand{\ep}{\varepsilon}

\newcommand{\wt}{\widetilde}
\newcommand{\wh}{\widehat}
\newcommand{\into}{\hookrightarrow}

\newcommand{\id}{\mathrm{id}}
\newcommand{\pt}{\mathrm{pt}}

\DeclareMathOperator{\Hom}{Hom}

\newcommand{\conv}{{\mathbin{\scalebox{1.1}{$\mspace{1.5mu}*\mspace{1.5mu}$}}}}
\newcommand{\Loc}{\mathrm{Loc}}
\newcommand{\Sh}{\mathrm{Sh}}
\newcommand{\Coh}{\mathrm{Coh}}
\newcommand{\Pic}{\mathrm{Pic}}

\newcommand{\msp}{ss}

\newcommand{\Mod}{\mathrm{Mod}}

\newcommand{\supp}{\mathrm{supp}}

\newcommand{\smint}{I}
\newcommand{\lgint}{\widetilde{I}}

\newcommand{\Int}{\mathrm{Int}}
\newcommand{\Relint}{\mathrm{Relint}}
\DeclareMathOperator{\cHom}{\mathscr{H}\text{\kern -3pt {\calligra\large om}}\,}
\newcommand{\ad}{\mathrm{ad}}
\newcommand{\Ad}{\mathrm{Ad}}
\newcommand{\ol}{\overline}
\newcommand{\coeff}{\bC}
\newcommand{\Conv}{\mathrm{Conv}}
\newcommand{\no}{}

\newcommand{\Div}{\mathrm{Div}}
\newcommand{\IndCoh}{\mathrm{IndCoh}}
\newcommand{\Perf}{\mathrm{Perf}}

\DeclareFontFamily{U}{mathx}{\hyphenchar\font45}
\DeclareFontShape{U}{mathx}{m}{n}{
	<5> <6> <7> <8> <9> <10>
	<10.95> <12> <14.4> <17.28> <20.74> <24.88>
	mathx10
}{}
\DeclareSymbolFont{mathx}{U}{mathx}{m}{n}
\DeclareFontSubstitution{U}{mathx}{m}{n}
\DeclareMathAccent{\widecheck}{0}{mathx}{"71}

\DeclareMathSymbol{\shortminus}{\mathbin}{AMSa}{"39}
\DeclareMathSymbol{\sm}{\mathbin}{AMSa}{"39}

\newcommand{\arrtip}{latex'}

\begin{document}
\title[Contact isotopies in the coherent-constructible correspondence]{Contact isotopies in the coherent-constructible correspondence}

\author[Jishnu Bose]{Jishnu Bose}
\address[Jishnu Bose]{University of Southern California \\ Los Angeles CA, USA}
\email{jishnubo@usc.edu}

\author[Harold Williams]{Harold Williams}
\address[Harold Williams]{University of Southern California \\ Los Angeles CA, USA}
\email{hwilliams@usc.edu}

\begin{abstract} 
The coherent-constructible correspondence is a realization of toric mirror symmetry in which the A-side is modeled by constructible sheaves on $T^n$. This paper provides a geometric realization of the mirror Picard group action in this correspondence, characterizing it in terms of quantized contact isotopies and providing a sheaf-theoretic counterpart to work of Hanlon in the Fukaya-Seidel setting.  Given a toric Cartier divisor~$D$, we consider a family of homogeneous Hamiltonians~$H_\ep$ on $\dot{T}^* T^n$. Their flows act on sheaves via a family of kernels $K_\ep$ on $T^n \times T^n$. The nearby cycles kernel $K_0$ corresponds heuristically to the Hamiltonian flow of the non-differentiable function $\lim_{\ep \to 0} H_\ep$, which is the pullback of the support function of~$D$ along the cofiber projection. We show that the action of $K_0$ coincides with the convolution action of the associated twisted polytope sheaf, hence mirrors the action of~$\cO(D)$ on coherent sheaves. 
\end{abstract}

\maketitle

\setcounter{tocdepth}{1}

\tableofcontents

\section{Introduction}
\thispagestyle{empty}

Let $X_\Sigma$ be the complex toric variety associated to a fan $\Sigma \subset \bR^n$. Mirror symmetry relates the derived category $\Coh(X_\Sigma)$ of coherent sheaves to the symplectic geometry of~$T^* T^n$, a connection which has been formalized in various ways. One of these is the coherent-constructible correspondence \cite{Bon06,FLTZ11,Tre10,Kuw20}, or CCC, an equivalence
\begin{equation*}\label{eq:ccc}
	\Coh(X_\Sigma) \cong \Sh_{\Lambda_\Sigma}^c(T^n). 
\end{equation*} 
Here the right side is the category of compact weakly constructible sheaves whose singular support meets the cosphere bundle $S^* T^n$ along a singular Legendrian $\Lambda_\Sigma$. This formalization complements those involving different variants of the Fukaya category of $T^* T^n$ \cite{Sei01,AKO06,Abo06,Abo09,KK17,HH22}, and reflects the close relationship between the sheaf theory of manifolds and the symplectic geometry of cotangent bundles \cite{KS94,NZ09,GPS18,Gui23}. 

Through its equivalence with $\Coh(X_\Sigma)$, the category $\Sh_{\Lambda_\Sigma}^c(T^n)$ inherits an action of the Picard group of $X_\Sigma$. A description of this mirror action is provided by the fact that the CCC intertwines tensoring on $X_\Sigma$ with convolution on $T^n$. The action of a line bundle $\cO(D)$ is thus given by convolution with its mirror, the pushforward of a certain sheaf $\cP(D)$ on $\bR^n$ along the projection $p: \bR^n \to T^n$. This $\cP(D)$ is called a twisted polytope sheaf \cite{Zho19}, as its support is a twisted polytope in the sense of \cite{KT93}. It may be presented as a complex of shard sheaves, certain conical sheaves which are mirror to summands of the \v{C}ech resolution of~$\cO(D)$. 

Though combinatorially explicit, this description is unsatisfying in that it has no direct relationship with the symplectic geometry of $T^* T^n$. By contrast, under suitable hypotheses the work of Hanlon \cite{Han19} provides an elegant geometric description of the mirror Picard group action on a variant of the Fukaya category, the monomially admissible Fukaya-Seidel category. Specifically, the mirror action of $\cO(D)$ is induced by the Hamiltonian flow associated to a smoothing of the support function of $D$. The goal of this paper is to formulate and prove a similar result in the setting of the CCC, in fact obtaining one which holds for more general~$X_\Sigma$. 

Our starting point is the work of Guillermou-Kashiwara-Schapira \cite{GKS12}. Given a manifold $M$, they prove that a homogeneous Hamiltonian isotopy of $\dot{T}^* M = T^* M \smallsetminus M$, or equivalently a contact isotopy of $S^* M$, is quantized by a canonical kernel $K \in \Sh(M \times M)$. Passing from a sheaf $\cF \in \Sh(M)$ to its singular support $\dot{\msp}(\cF) \subseteq \dot{T}^* M$ intertwines the underlying isotopy with the action of $K$ on $\Sh(M)$ by integral transforms. 

We apply this construction to a family of homogeneous Hamiltonians $H_{D,\ep}$ on $\dot{T}^* T^n$. Recall that $D$ is encoded by its support function $\varphi_D$, a continuous but only piecewise smooth function that is linear on the cones of $\Sigma$. We extend $\varphi_D$ to all of $\bR^n$ if $\Sigma$ is not complete, and then further define a family of smoothings $\varphi_{D,\ep}$ by convolving with a mollifier of radius~$\ep$, restricting to the unit sphere, and extending homogeneously. We then define $H_{D,\ep}$ as the pullback of $\varphi_{D,\ep}$ along the fiber projection $T^* T^n \to \bR^n$. 

By \cite{GKS12}, the flows of the $H_{D,\ep}$ are now quantized by a family of kernels $K_{D,\ep} \in \Sh(T^n \times T^n)$. While the non-smooth function $\lim_{\ep \to 0} H_{D,\ep}$ does not have a well-defined Hamiltonian flow, there is still a well-defined limit of the kernels $K_{D,\ep}$ in $\Sh(T^n \times T^n)$. Namely, we can assemble them into a total sheaf $K_{D,\smint} \in \Sh(T^n \times T^n \times \smint)$, where $\smint = (0,a)$ for some $a > 0$, and then form the nearby cycles kernel $\psi(K_{D,\smint}) \in \Sh(T^n \times T^n)$. By definition $\psi(K_{D,\smint}) \cong i^* j_* K_{D,\smint}$, where 
\[
\begin{tikzpicture}[baseline=(current bounding box.center), thick]
	\node (A) at (0,0) {$T^n \times T^n \times \{0\}$};
	\node (B) at (4,0) {$T^n \times T^n \times \lgint$};
	\node (C) at (8,0) {$T^n \times T^n \times \smint$};
	
	\draw[right hook->] (A) -- node[above, font=\scriptsize] {$i$} (B);
	\draw[left hook->] (C) -- node[above, font=\scriptsize] {$j$} (B);
\end{tikzpicture}
\]
and $\lgint$ is any open interval containing both $\smint$ and $0$. 

Our main result is the following relationship between this limit of quantized contact isotopies and the CCC. In the first equation, the left side refers to the action of kernels by integral transforms, and the right side to the convolution product on the group $T^n$. In the second equation, we write $A$ and $B$ for the inverse functors comprising the CCC. 

\begin{Theorem}\label{thm:mainthmintro}
Let $D$ be a torus-invariant Cartier divisor on $X_\Sigma$. For any $\cF \in \Sh(T^n)$ we have an isomorphism
\begin{equation*}\label{eq:nonequivmainthmintro} \psi(\no{K}_{D,\smint}) \circ \cF \cong (p_! \cP(D)) \conv \cF. \end{equation*}
In particular, if $\cF \in \Sh_{\Lambda_\Sigma}^c(T^n)$ we have an isomorphism
$$ \psi(\no{K}_{D,\smint}) \circ \cF \cong A(\cO(D) \otimes B(\cF)). $$
\end{Theorem}

As alluded to above, Theorem \ref{thm:mainthmintro} is a counterpart of \cite[Thm. 4.5]{Han19} in the setting of (monomially admissible) Fukaya-Seidel categories and its corollary \cite[Thm. 3.13]{HH22} in the setting of partially wrapped Fukaya categories. However, its statement does not map exactly onto those of the latter results, which contain no direct analogue of nearby cycles. It is monomial admissibility which in some sense avoids this, as it restricts the class of Lagrangians included so that for finite $\ep$ they are stable under the flow of a Hamiltonian related to $H_{D,\ep}$. But $\Lambda_\Sigma$ is not preserved by this flow for any finite $\ep$, which leads to the limiting aspect of Theorem \ref{thm:mainthmintro}. 

Note that while $X_\Sigma$ is an arbitrary normal toric variety in Theorem \ref{thm:mainthmintro}, it is required to be smooth and projective in \cite[Thm. 4.5]{Han19} and \cite[Thm. 3.13]{HH22}. While the simplicial, quasi-projective case is not substantially different, new technical challenges seem to stand in the way of extending these Floer-theoretic results to the case when $X_\Sigma$ has arbitrary singularities, or has no ample line bundles (e.g. \cite[Rem. 3.9]{HH22}). 

Recall that \cite{GPS18} establishes an equivalence between the sheaf and partially wrapped Fukaya categories associated to a singular Legendrian in a cosphere bundle. When $X_\Sigma$ is smooth and projective, we could imagine using an extension of \cite{GPS18} to deduce some version of Theorem \ref{thm:mainthmintro} from \cite[Thm. 3.13]{HH22}.  This would require analyzing how \cite[Thm. 1.1]{GPS18} interacts with limits and nearby cycles, as well as with the construction of \cite{GKS12}. On the other hand, part of the point of Theorem \ref{thm:mainthmintro} is exactly that its statement and proof are independent of Floer theory. We take the view that sheaf theory and Floer theory provide complementary perspectives on the A-side, and that mirror symmetry is understood most deeply when its different aspects are understood independently from both of them. 

We also remark that the proof of the CCC for simplicial fans in \cite{She22} offers yet another perspective on the mirror Picard group action. Here the Legendrian $\Lambda_\Sigma$ arises by construction as a slice of a larger Legendrian living over a larger torus. Parallel transport along the additional directions in this torus yields a monodromy action on $\Sh_{\Lambda_\Sigma}(T^n)$, which can be identified with the mirror action of $\Pic(X_\Sigma)$ \cite[Rem. after Thm. 17]{She22}. 

The remainder of the paper is organized as follows. In Section \ref{sec:generalities} we review some relevant background and fix our notation. Section \ref{sec:GKSonLie} collects those results we establish which are ultimately general facts about GKS kernels on a Lie group $G$. The first main result, Proposition~\ref{prop:GKSongroups}, states that the action of a GKS kernel associated to a bi-invariant Hamiltonian on $\dot{T}^* G$ can be expressed in terms of group convolution. The second, Proposition \ref{prop:nearbyactions}, provides criteria for the action of a family of such kernels to commute with nearby cycles, compatibly with the isomorphisms provided by Proposition \ref{prop:GKSongroups}. While our application uses these results when~$G$ is abelian, another case we find interesting is when $G$ is compact and the Hamiltonian is the norm associated to a bi-invariant metric. 

Section \ref{sec:mainproofs} contains the proof of Theorem \ref{thm:mainthmintro}. It begins with a detailed analysis of how the derivatives of the smoothings $\varphi_{D,\ep}$ behave as $\ep \to 0$ (Propositions \ref{prop:limdphiepxi}, \ref{prop:limsup}, and \ref{prop:dvarphiepbounded}). We use these to bound the singular support of the action of $\psi(K_{D,\smint})$ on a shard sheaf, and show moreover that the resulting sheaf is essentially determined by its singular support (Propositions \ref{prop:sigma1perpssprop} and \ref{prop:psiKDonshards}). Using this we deduce the complete, equivariant version of Theorem \ref{thm:mainthmintro} (Theorem \ref{thm:equivmainthm}), and analyze how it interacts with the projection $\bR^n \to T^n$ to deduce Theorem \ref{thm:mainthmintro} in the complete case (Theorem \ref{thm:nonequivmainthm}). We then deduce the noncomplete case from this in Theorem \ref{thm:noncompletemainthm}. 

\subsection*{Acknowledgements}
We thank Sheel Ganatra, Andrew Hanlon, Christopher Kuo, Wenyuan Li, and Eric Zaslow for valuable comments and discussions. H. W. was supported by NSF grant DMS-2143922. 

\section{Generalities}\label{sec:generalities}

We collect here some general results we will refer back to as needed, as well as summarizing the terminology and notation we will use. 

\subsection{Sheaves and spaces} By default, category will mean $\infty$-category, but the reader may safely ignore this (i.e. most or all arguments apply equally to derived $\infty$-categories and their underlying triangulated categories). Given a ring $R$, we write $\Mod_R$ for the derived $\infty$-category of $R$-modules. Its homotopy category is the classical unbounded derived category of $R$-modules. The heart $\Mod_R^\heartsuit$ of its natural t-structure is the ordinary category of $R$-modules. 

By default, space will mean locally compact Hausdorff space. We write $\Sh(X)$ for the category of $\Mod_\bC$-valued sheaves on a space $X$, and $\Loc(X)$ for its subcategory of locally constant sheaves. These categories have natural t-structures whose hearts $\Sh(X)^\heartsuit$ and $\Loc(X)^\heartsuit$ are the ordinary categories of $\Mod_\bC^{\heartsuit}$-valued sheaves and local systems. The category $\Loc(X)$ is compactly generated and we write $\Loc^c(X)$ for its subcategory of compact objects. 

To a continuous map $f: X \to Y$ are associated functors $f^*, f^!: \Sh(Y) \to \Sh(X)$ and $f_*,f_!: \Sh(X) \to \Sh(Y)$. Along with tensor and sheaf Hom, these functors extend their classical counterparts defined at the level of homotopy categories with suitable boundedness conditions. We refer to \cite[Ch. 2]{KS94} for the classical theory and \cite{Jin24,Vol21} for details on its unbounded, $\infty$-categorical extension. 

We write $\pt_X: X \to \pt$ for the projection, and have the constant and dualizing sheaves $\bC_X := \pt_X^* \bC$ and $\omega_X := \pt_X^! \bC$. We have the Verdier dual $\bD_X:= \cHom(-,\omega_X)$, which we write as $\bD$ when $X$ is clear from context. If $i: X \into Y$ is the inclusion of a subspace, we sometimes write $\bC_X$ for $i_! \bC_X$ when it is clear we are referring to a sheaf on $Y$. 

Let $\smint \subseteq \bR_+$ be an open interval whose closure in~$\bR$ contains $0$, and let $\lgint \subseteq \bR$ be an open interval containing $\{0\}$ and $\smint$. We denote their inclusions by $i: \{0\} \into \lgint$ and $j: \smint \into \lgint$. Given a space $X$, the associated nearby cycles functor $\psi$ is then the composition
$$ \psi: \Sh(X \times \smint) \xrightarrow{\ol{j}_*} \Sh(X \times \lgint) \xrightarrow{\ol{i}^*} \Sh(X). $$
Here and elsewhere, we adopt the convention that if $f$ is a map of spaces, $\ol{f}$ is the product of~$f$ with an identity map which is clear from context (i.e. $\ol{i} = \id_X \times i $ above). 

\subsection{Kernels and integral transforms}
Given a space $X$, the integral transform associated to a kernel $K \in \Sh(X \times X)$ is the functor $ K \circ - : \Sh(X) \to \Sh(X)$ given by 
$$ K \circ - := \pi_{1!}(K \otimes \pi_2^*(-)), $$
where $\pi_1$ and $\pi_2$ are the projections to the left and right factors. We will also consider the following relative version of this construction. To $K \in \Sh(X \times X \times Y)$ we associate the functor $ K \circ_Y - : \Sh(X \times Y) \to \Sh(X \times Y)$ given by 
$$ K \circ_Y - := \ol{\pi}_{1!}(K \otimes \ol{\pi}_2^*(-)), $$
where $\ol{\pi}_1 = \pi_1 \times \id_Y$, similarly for $\ol{\pi}_2$. 

For integral transforms, we have the following analogue of the projection formula. 

\begin{Proposition}\label{prop:convprojform}
	Let $f: X \to Y$ be a map of spaces, and let $\ol{f}_1 = f \times \id_X$ and $\ol{f}_2 = \id_Y \times f$. Suppose that $K \in \Sh(X \times X)$ and $\wt{K} \in \Sh(Y \times Y)$ are kernels such that $\ol{f}_{1!} K \cong \ol{f}_2^* \wt{K}$. Then for any $\cF \in \Sh(X)$ there is an isomorphism
	$$ \wt{K} \circ f_! \,\cF \cong f_!(K \circ \cF). $$
\end{Proposition}
\begin{proof}
	The desired isomorphism is obtained as the composition
	\begin{align*}
		\wt{K} \circ f_! \,\cF & \cong \pi_{1!}(\wt{K} \otimes \pi_2^* f_! \cF)\\ 
		& \cong \pi_{1!}(\wt{K} \otimes \ol{f}_{2!} \pi_2^* \cF)\\
		& \cong \pi_{1!}\ol{f}_{2!} (\ol{f}_{2}^* \wt{K} \otimes  \pi_2^* \cF)\\
		& \cong \pi_{1!} (\ol{f}_{1!} K \otimes  \pi_2^* \cF)\\
		& \cong \pi_{1!}\ol{f}_{1!} ( K \otimes  \pi_2^* \cF)\\
		& \cong f_! \pi_{1!} ( K \otimes  \pi_2^* \cF)\\
		& \cong f_!(K \circ \cF).
	\end{align*}
	Here the second isomorphism uses proper base change, the third uses the projection formula, the fourth uses our hypothesis and  $\pi_1 = \pi_1 \circ \ol{f}_2$, the fifth uses the projection formula and $\pi_2 = \pi_2 \circ \ol{f}_1$, and the sixth uses $f \circ \pi_1 = \pi_1 \circ \ol{f}_1$. 
\end{proof}

\subsection{Manifolds and constructibility}
We mostly follow the conventions of \cite[Ch. 8]{KS94} regarding manifolds and constructible sheaves. By default, manifold will mean real analytic manifold. If $M$ is a manifold, we write $\dot{T}^* M$ for the complement of the zero section in  $T^* M$ and $S^* M$ for the cosphere bundle of $M$. We freely identify $S^* M$ both as the quotient of~$\dot{T}^* M$ by $\bR_+$, and as the subset of $T^* M$ formed by the unit covectors associated to a metric on $M$. If $X \subset \dot{T}^*M$ is a conic subset, we write $[X] \subseteq S^* M$ for its image under the quotient map, or equivalently its intersection with the set of unit covectors.

A sheaf on $M$ is constructible if it is bounded, has perfect stalks, and has locally constant restrictions to the strata of a locally finite subanalytic stratification \cite[Def. 8.4.3, Thm. 8.3.20]{KS94}. More generally, we say it is locally constructible if its restrictions to some open cover are constructible, and weakly constructible if the third condition above holds. Note that the first condition follows from the second if the stratification in the third condition can be taken to be finite. 

Associated to $\cF \in \Sh(M)$ is its singular support $\msp(\cF) \subset T^* M$, a (typically singular) conic Lagrangian if $\cF$ is weakly constructible. We also write $\dot{\msp}(\cF) \subseteq \dot{T}^* M$ for $\msp(\cF) \cap \dot{T}^* M$ and $\msp^\infty(\cF) \subseteq S^* M$ for $[\dot{\msp}(\cF)]$. Given a subset $\Lambda \subseteq S^* M$ we write $\Sh_{\Lambda}(M) \subseteq \Sh(M)$ for the full subcategory of sheaves $\cF$ with $\msp^\infty(\cF) \subseteq \Lambda$. We write $\Sh_{\Lambda}^c(M)$ for its subcategory of compact objects, and $\Sh_{\Lambda}^b(M)$ for its subcategory of constructible objects. 

We will need the following result, whose proof we include for lack of an applicable reference. In the complex algebraic setting a proof appears in \cite[Prop. 2.9.1]{Ach21}. It can be adapted to the subanalytic setting, but we also have the following argument based on \cite[Prop. 3.4.4]{KS94}. Recall that a sheaf is (locally) cohomologically constructible if it is (locally) bounded and its stalks/costalks are perfect and computable via small enough open sets \cite[Def. 3.4.1]{KS94}. This is a slightly weaker notion than the subanalytic constructibility considered above. 

\begin{Proposition}\label{prop:boxpush}
	Let $f: X' \to X$ and $g: Y' \to Y$ be maps of spaces. If $\cF \in \Sh(X')$ and $\cG \in \Sh(Y')$ are locally cohomologically constructible, then there is a natural isomorphism
	\begin{equation}\label{eq:lowerstarboxtimes} (f \times g)_*(\cF \boxtimes \cG) \cong f_* \cF \boxtimes g_* \cG. \end{equation} 
\end{Proposition}
\begin{proof}
	Factoring $(f \times g)$ as $(f \times \id_Y) \circ (\id_{X'} \times g)$, it suffices to consider the case $X' = X$, $f = \id_X$. We obtain (\ref{eq:lowerstarboxtimes}) as the composition
	\begin{align*}
		(\id_X \times g)_*(\cF \boxtimes \cG) & \cong (\id_X \times g)_* \cHom(\pi_{X}^* \bD \cF, \pi_{Y'}^! \cG) \\
		& \cong (\id_X \times g)_* \cHom((\id_X \times g)^*\pi_{X}^* \bD \cF, \pi_{Y'}^! \cG) \\
		& \cong \cHom(\pi_{X}^* \bD \cF, (\id_X \times g)_*\pi_{Y'}^! \cG) \\
		& \cong \cHom(\pi_{X}^* \bD \cF, \pi_{Y}^! g_* \cG) \\
		& \cong \cF \boxtimes g_* \cG. 
	\end{align*}
	Here the first and last isomorphisms follow from our hypotheses on $\cF$ and $\cG$ and from \cite[Prop. 3.4.4]{KS94} (we note that its proof only uses local cohomological constructibility). The second is trivial, the third is \cite[Eq. 2.6.15]{KS94}, and the fourth follows by base change. 
\end{proof}

\subsection{Quantized contact isotopies}

We summarize here some results of \cite{GKS12}, slightly adapted to the form in which we will use them. Let $M$ and $U$ be manifolds and suppose~$U$ is contractible. Let $H: \dot{T}^* M \times U \to \bR$ be a smooth function whose restriction $H_u$ to $\dot{T}^* M \times \{u\}$ is homogeneous for all $u \in U$, in the sense that $H_u$ intertwines the $\bR_+$-actions on $\dot{T}^* M$ and~$\bR$. For all $u \in U$ and $t \in \bR$ the time-$t$ Hamiltonian flow $\Phi_{u,t}$ of $H_u$ is then homogeneous, in the sense of commuting with the $\bR_+$-action. 

By \cite[Thm. 3.7, Rem. 3.9]{GKS12}, \cite[Rem. 2.1.2(1)]{Gui23} there is then a unique kernel $K \in \Sh(M \times M \times U \times \bR)$ such that 
\begin{enumerate}
	\item $\ol{i}_0^* K \cong \bC_{\Delta_M \times U}$, where $i_0$ is the inclusion $\{0\} \into \bR$, 
	\item the projection $\dot{\msp}(K) \to \ol{\pi}_{U \times \bR} (\dot{\msp}(K))$ is a diffeomorphism, where $\pi_{U \times \bR}: \dot{T}^*(U \times \bR) \to U \times \bR$ is the projection, and
	\item we have
	\begin{equation*}
		\ol{\pi}_{U \times \bR} (\dot{\msp}(K)) = \{(\Phi_{u,t}(m,\xi), m,\sm \xi, u, t) \}_{(m, \xi, u, t) \in \dot{T}^* M \times U \times \bR} \subseteq \dot{T}^* (M \times M) \times U \times \bR. 
	\end{equation*}
\end{enumerate}
This $K$ is referred to as a GKS kernel, and we sometimes also write it as $K_{U \times \bR}$. By \cite[Prop. 3.12]{GKS12}, the action of $K$ on sheaves is intertwined with the given Hamiltonian flows by the formation of singular support. That is, for any $\cF \in \Sh^b(M)$ and $(u,t) \in U \times \bR$ we have 
$$ \dot{\msp}(K_{u,t} \circ \cF) = \Phi_{u,t}(\dot{\msp}(\cF)), $$
where $K_{u,t}$ is the restriction of $K$ to $M \times M \times \{(u,t)\}$. 

\subsection{Toric varieties}
We mostly follow the conventions of \cite{CLS11} regarding toric geometry. Throughout, we fix an integer $n \geq 0$ and a lattice $N$ of rank $n$, and write $M\cong \Hom_\bZ(N, \bZ)$ for the dual lattice. We have associated vector spaces $N_\bR := N\otimes_\bZ \bR$ and $M_\bR := M\otimes_\bZ \bR$, and the associated algebraic torus $T_N := N \otimes_\bZ (\bC)^\times$. We generally write elements of $N_\bR$ as~$\xi$, since later we usually identify $N_\bR$ with the cotangent space to a point of $M_\bR$. 

We write $X_\Sigma$ for the complex toric variety associated to a fan $\Sigma$ in $N_\bR$. To each ray $\rho \in \Sigma(1)$ is associated a $T_N$-invariant prime divisor $D_\rho$. Together these freely generate the group $\Div_{T_N}(X_\Sigma)$ of $T_N$-invariant divisors on $X_\Sigma$. There is a natural homomorphism $M \to \Div_{T_N}(X_\Sigma)$, and the class group of $X_\Sigma$ is the quotient of $\Div_{T_N}(X_\Sigma)$ by its image \cite[Thm. 4.1.3]{CLS11}. 

We can characterize which $T_N$-invariant divisors are Cartier as follows. First, for each cone~$\sigma$ we write $\sigma(1)$ for its set of rays, and for each ray $\rho$ we write $u_\rho$ for its primitive generator. Then $\sum_{\rho \in \Sigma(1)} a_\rho D_\rho$ is Cartier if and only if there is a collection $\{m_\sigma\}_{\sigma \in \Sigma}$ of elements of $M$ such that 
$$\langle m_\sigma\,|\,u_\rho \rangle =  \sm a_\rho$$
for all $\sigma$ and all $\rho \in \sigma(1)$ \cite[Thm. 4.2.8]{CLS11}. The collection $\{m_\sigma\}_{\sigma \in \Sigma}$ is referred to as Cartier data for $D$. For fixed $D$, each $m_\sigma$ is uniquely determined up to translation by $\sigma^\perp \cap M$, and in particular $m_\sigma$ is uniquely determined if $\sigma$ is maximal. 

Later it will be convenient to generalize the above notion slightly. We say that a collection $\{m_\sigma\}_{\sigma \in \Sigma}$ of elements of $M_\bR$ is $\bR$-Cartier data for $D$ if it satisfies the same condition. Note that $D$ may still admit $\bR$-Cartier data even if it is not Cartier. If $D$ is Cartier, however, then for any $\bR$-Cartier data and any maximal cone $\sigma$ the element $m_\sigma$ is still uniquely determined and belongs to $M$. 

The support function of a Cartier divisor $D$ is the continuous function $\varphi_D: |\Sigma| \to \bR$ given~by 
$$\varphi_D(\xi) = \langle m_\sigma \,|\, \xi \rangle$$
for any $\sigma$ containing $\xi$ and any Cartier data for $D$. In practice $D$ will usually be fixed, hence often we simply write $\varphi$ for $\varphi_D$. 

\subsection{The coherent-constructible correspondence}\label{sec:CCC}
Next we review the main constructions of \cite{FLTZ11} and \cite{Kuw20}.  We write $T^n$ for the compact torus $M_\bR/M$, implicitly choosing a basis of $M_\bR$ to do so, and write $p: M_\bR \to T^n$ for the projection. We freely use the trivializations
\begin{align*}
T^* M_\bR \cong M_\bR \times N_\bR, &\quad  T^* T^n \cong T^n \times N_\bR \\
S^* M_\bR \cong M_\bR \times [\dot{N}_\bR], &\quad  S^* T^n \cong T^n \times [\dot{N}_\bR],
\end{align*}
where $\dot{N}_\bR := N_\bR \smallsetminus \{0\}$ and where $[\dot{N}_\bR]$ is the quotient of $\dot{N}_\bR$ by $\bR_+$. 

Associated to a fan $\Sigma$ we have the Legendrian
$$ \wt{\Lambda}_\Sigma := \bigcup_{\sigma \in \Sigma} (\sigma^\perp + M) \times [\sm \dot{\sigma}] \subseteq S^* M_\bR. $$
When $\Sigma$ is complete, the equivariant CCC of \cite[Main Thm.]{FLTZ11} is a pair of inverse equivalences
$$ \wt{A}: \Perf^{T_N}(X_\Sigma) \rightleftarrows \Sh_{\wt{\Lambda}_\Sigma}^{b,cpt}(M_\bR) :\wt{B}.$$ 
Here $\Perf^{T_N}(X_\Sigma)$ is the derived category of $T_N$-equivariant perfect complexes on $X_\Sigma$, and $\Sh_{\wt{\Lambda}_\Sigma}^{b,cpt}(M_\bR) \subseteq \Sh_{\wt{\Lambda}_\Sigma}^b(M_\bR)$ is the subcategory of objects with compact support. The above equivalences intertwine the tensor product on $\Perf^{T_N}(X_\Sigma)$ with the convolution product on $\Sh_{\wt{\Lambda}_\Sigma}^{b,cpt}(M_\bR)$ \cite[Cor. 3.13]{FLTZ11}. Recall that if $G$ is a Lie group, the convolution product on $\Sh(G)$ is given by 
$$ \cF \conv \cG \cong m_!(\pi_1^*\cF \otimes \pi_2^* \cG),$$
where $m: G \times G \to G$ is the multiplication map and $\pi_1, \pi_2$ are the projections. 

Given a Cartier divisor $D = \sum_{\rho \in \Sigma(1)} a_\rho D_\rho$, the line bundle $\cO(D)$ has a natural $T_N$-equivariant structure. The sheaf $\cP(D) := \wt{A}(\cO(D))$ is referred to as a twisted polytope sheaf \cite{Zho19}. Applying $\wt{A}$ to the \v{C}ech resolution of $\cO(D)$, one obtains a presentation of $\cP(D)$ as a complex of the form 
$$   \cdots \to \bigoplus_{\sigma \in \Sigma(k+1)} \bC_{\Int(\sigma^\vee + \chi_\sigma )} \to \bigoplus_{\sigma \in \Sigma(k)} \bC_{\Int(\sigma^\vee + \chi_\sigma )} \to \cdots.  $$
Here $\Int(\sigma^\vee + \chi_\sigma )$ is the interior of $\sigma^\vee + \chi_\sigma$,
and $\bC_{\Int(\sigma^\vee + \chi_\sigma )}$ is in degree $\sm k$ for $\sigma \in \Sigma(k)$. Note that $\sigma^\vee + \chi_\sigma  \subseteq \tau^\vee + \chi_\tau $ if $\tau$ is a facet of $\sigma$. The differentials in the above complex are sums of the induced maps $\bC_{\Int(\sigma^\vee + \chi_\sigma )} \to \bC_{\Int(\tau^\vee + \chi_\tau)},$ twisted by suitable signs \cite[Rem. 3.10]{Zho20} (note that our conventions follow those of \cite{FLTZ11}, which are Verdier dual to those of \cite{Zho20}). 

  We also have the Legendrian
$$ \Lambda_\Sigma := \bigcup_{\sigma \in \Sigma} p(\sigma^\perp) \times [\sm \dot{\sigma}] \subseteq S^* T^n. $$
The nonequivariant CCC of \cite[Thm 1.2]{Kuw20} is a pair of inverse equivalences 
$$ A: \Coh(X_\Sigma) \rightleftarrows \Sh_{\Lambda_\Sigma}^c(T^n) :B.$$ 
This equivalence intertwines the $!$-tensor product on the ind-completion $\IndCoh(X_\Sigma)$ of $\Coh(X_\Sigma)$ with the convolution product on the ind-completion $\Sh_{\Lambda_\Sigma}(T^n)$ of $\Sh_{\Lambda_\Sigma}^c(T^n)$ \cite[Rem. 12.12]{Kuw20}. Recall that the $!$-tensor product of $\cF, \cG \in \IndCoh(X_\Sigma)$ is 
$$\cF \overset{!}{\otimes} \cG :=\Delta^!(\cF \boxtimes \cG),$$ where $\Delta: X_\Sigma \to X_\Sigma \times X_\Sigma$ is the diagonal. 

The equivariant and nonequivariant CCCs are related as follows. We write $q^*$ for the forgetful functor $\Coh^{T_N}(X_\Sigma) \to \Coh(X_\Sigma)$, this being equivalent to pullback along the stack quotient $q: X_\Sigma \to X_\Sigma/T_N$. There is then a commutative diagram
\begin{equation}\label{eq:equivnonequivdiagram}
	\begin{tikzpicture}
		[baseline=(current  bounding  box.center),thick,>=\arrtip]
		\node (a) at (0,0) {$\Perf^{T_N}(X_\Sigma)$};
		\node (b) at (4,0) {$\Sh_{\wt{\Lambda}_\Sigma}^{b,cpt}(M_\bR)$};
		\node (c) at (0,-1.5) {$\Coh(X_\Sigma)$};
		\node (d) at (4,-1.5) {$\Sh_{\Lambda_\Sigma}^c(T^n),$};
		\draw[->] (a) to node[above] {$\wt{A} $} (b);
		\draw[->] (b) to node[right] {$p_! $} (d);
		\draw[->] (a) to node[left] {$q^*(-) \otimes \omega_{X_\Sigma} $}(c);
		\draw[->] (c) to node[above] {$A $} (d);
	\end{tikzpicture}
\end{equation}
where $\omega_{X_\Sigma}$ is the dualizing complex of $X_\Sigma$. This follows by combining \cite[Thm 2.4]{Tre10}, which shows that $p_!$ intertwines $q^*$ with a certain functor $\kappa: \Perf(X_\Sigma) \to \Sh_{\Lambda_\Sigma}^b(T^n)$, and \cite[Lem. 14.5]{Kuw20}, which implies that $A(- \otimes \omega_{X_\Sigma}) \cong \kappa$.

\section{GKS kernels on Lie groups}\label{sec:GKSonLie}

In this section, we study some special features of GKS kernels when the underlying manifold is a Lie group. The first main result, Proposition \ref{prop:GKSongroups}, states that the action of a GKS kernel associated to a bi-invariant Hamiltonian can be expressed in terms of group convolution. The second, Proposition \ref{prop:nearbyactions}, provides criteria for the action of a family of such kernels to commute with nearby cycles, compatibly with the isomorphisms provided by Proposition \ref{prop:GKSongroups}. 

We fix throughout the section a Lie group $G$ with multiplication $m: G \times G \to G$ and identity $e \in G$. Given a space $X$, we consider the relative group convolution
$$ \cF \conv_X\, \cG :=  \ol{m}_{!}(\ol{\pi}_{1}^* \cF \otimes \ol{\pi}_{2}^* \cG)$$
on the category $\Sh(G \times X)$.  Here $\pi_1, \pi_2: G \times G \to G$ are the left and right projections, and we recall our convention that e.g. $\ol{m}$ denotes the product of $m$ with an identity map that is clear from context (i.e. $\ol{m} = m \times \id_X$ above). 

We begin with a preliminary result relating relative group convolution to the action of $\Sh(G \times G \times X)$ on $\Sh(G \times X)$ given by
$$ K \circ_X \cG :=  \ol{\pi}_{1!}(K \otimes \ol{\pi}_{2}^*\cG).$$ 
For both operations we will sometimes suppress a pullback to $G \times X$ when it is clear from context; that is, given $\cG \in \Sh(G)$ we will write $\cF \conv_X\, \cG$ for $\cF \conv_X\, \pi_G^*\, \cG$ and $K \circ_X \cG$ for  $K \circ_X \pi_G^*\, \cG$.

\begin{Proposition}\label{prop:groupconvidentity}
	Let $a: G \times G \to G$ be the map given by $a(g,h) = g h^{-1}$. Then for any space $X$ and any $\cF \in \Sh(G \times X)$, the kernel representing relative group convolution with $\cF$ is given by $\ol{a}^*\cF \in \Sh(G \times G \times X)$. That is, for any $\cG \in \Sh(G \times X)$ we have a natural isomorphism 
	$$ \cF \conv_X\, \cG \cong (\ol{a}^*\cF) \circ_X \cG. $$
\end{Proposition}
\begin{proof}
	Note that the maps
	$$\ol{m} \times_X \ol{\pi}_2, \ol{a} \times_X \ol{\pi}_2: G \times G \times X \to (G\times X) \times_X (G \times X) \cong G \times G \times X $$
	given by $(g, h, x) \mapsto (gh, h, x)$ and $(g, h, x) \mapsto (gh^{-1}, h, x)$, respectively, are inverse to each other. It follows that 
	\begin{align*}
		\cF \conv_X\, \cG & \cong \ol{m}_!(\ol{\pi}_1^* \cF \otimes \ol{\pi}_2^* \cG)\\
		& \cong \ol{\pi}_{1!} (\ol{m} \times_X \ol{\pi}_2)_! (\ol{\pi}_1^* \cF \otimes \ol{\pi}_2^* \cG)\\
		& \cong \ol{\pi}_{1!} (\ol{a} \times_X \ol{\pi}_2)^* (\ol{\pi}_1^* \cF \otimes \ol{\pi}_2^* \cG)\\
		& \cong \ol{\pi}_{1!} (\ol{a}^* \cF \otimes \ol{\pi}_2^* \cG)\\
		& \cong (\ol{a}^*\cF) \circ_X \cG \qedhere
	\end{align*}
\end{proof}

The main results of this section concern the flow induced by a homogeneous, bi-invariant Hamiltonian on the complement of the zero section in $T^* G$, or a family of such flows. In our intended application $G$ is abelian, so bi-invariance is a trivial condition. A natural nonabelian example is provided by the norm associated to a bi-invariant metric when $G$ is compact. We begin by showing that such a flow is described by a simple exponential formula. 

\begin{Proposition}\label{prop:GKSongroupsflowpart}
	Let $H:\dot{T}^* G \to \bR$ be a $G\times G$-invariant smooth function and $\varphi: \dot{\fg}^* \to \bR$ its restriction to $\dot{T}^*_e G$. Then in the right-invariant trivialization $\dot{T}^* G \cong G \times \dot{\fg}^*$, the time-$t$ Hamiltonian flow of $H$ is given by 
	$$ (g, \xi) \mapsto (e^{t d\varphi(\xi)}g, \xi), $$
	where we interpret $d\varphi$ as a map $\dot{\fg}^* \to \fg$. 
\end{Proposition}
\begin{proof}
Write $X_{H}$ for the Hamiltonian vector field of $H$. Using the right-invariant trivialization $T (G \times \dot{\fg}^*) \cong (G \times \dot{\fg}^*) \times (\fg \times \fg^*)$, we interpret $X_{H}$ as a smooth map $G \times \dot{\fg}^* \to \fg \times \fg^*$. Recall that in this trivialization the symplectic form on $T_{(g,\xi)}(G \times \dot{\fg}^*) \cong \fg \times \fg^*$ is given by 
$$ \omega((v_1, \theta_1), (v_2, \theta_2)) = \langle v_1 | \theta_2 \rangle - \langle v_2 | \theta_1 \rangle + \langle [v_1,v_2] | \xi \rangle $$
for any $(g, \xi) \in G \times \dot{\fg}^*$. Since $H$ is right invariant we have $H = \varphi \circ \pi_{\dot{\fg}^*}$, where $\pi_{\dot{\fg}^*}: \dot{T}^* G \to \dot{\fg}^*$ is the projection to nonzero right-invariant one-forms. We thus have $d H(g,\xi) = (0, d \varphi(\xi))$, hence the above expression for the symplectic form implies
\begin{equation}\label{eq:XHm} X_{H}(g, \xi) = (d \varphi(\xi), \ad^*_{d \varphi(\xi)} \xi). \end{equation}

Given that $H$ is right-invariant, the further condition that it is bi-invariant is equivalent to the condition that $\varphi$ is coadjoint-invariant. For any $v \in \fg$ and $\xi \in \fg^*$ we then have
$$ \langle v | \ad^*_{d\varphi(\xi)} \xi \rangle = \sm \langle d\varphi(\xi) | \ad^*_v \xi \rangle = \sm \frac{d}{dt} \varphi(\Ad^*_{e^{tv}} \xi)|_{t=0} = 0, $$
the second equality following from the chain rule and the third from coadjoint-invariance. 
But since~$v$ was arbitrary, it follows that 
\begin{equation}\label{eq:vanishingidentity}
	\ad^*_{d \varphi_m(\xi)} \xi = 0
\end{equation} 
for all $\xi \in \fg^*$. Since the flow along a right-invariant vector field on $G$ is given by left multiplication by the corresponding one-parameter subgroup, the claim now follows from (\ref{eq:XHm}) and (\ref{eq:vanishingidentity}). 
\end{proof}

Next we consider the GKS kernel associated to a homogeneous, bi-invariant Hamiltonian, showing that its action can be described in terms of group convolution.  

\begin{Proposition}\label{prop:GKSongroups}
	Let $M$ be a contractible manifold and $H:\dot{T}^* G \times M \to \bR$ a homogeneous, $G\times G$-invariant smooth function, where $\bR_+$ and $G \times G$ act trivially on $M$. Let $K \in \Sh(G \times G \times M \times \bR)$ be the associated GKS kernel, and set $\cP := K \circ_{M \times \bR} \bC_{e} \in \Sh(G \times M \times \bR)$. Then there is a natural isomorphism
	\begin{equation}\label{eq:kernelsheafid} K \circ_{M \times \bR} \cF \cong \cP \conv_{M \times \bR}\, \cF\end{equation}
	for all $ \cF \in \Sh(G \times M \times \bR)$. 
\end{Proposition}

\begin{proof}
Throughout we write $H_m$ and $\varphi_m$ for the restrictions of $H$ to $\dot{T}^* G \times \{m\}$ and $\dot{T}^*_e G \times \{m\}$, and interpret the differential $d \varphi_m$ as a smooth map $\dot{\fg}^* \to \fg$. By Proposition \ref{prop:groupconvidentity}, the desired isomorphism (\ref{eq:kernelsheafid}) will follow if we show that  $K \cong \ol{a}^* \cP$, where $a: G \times G \to G$ is given by $a(g, h) = gh^{-1}$. By \cite[Thm. 3.7, Rem. 3.9]{GKS12}, \cite[Rem. 2.1.2(1)]{Gui23}, and Proposition~\ref{prop:GKSongroupsflowpart}, the kernel $K$ is uniquely determined by the conditions that
\begin{enumerate}
	\item $\ol{i}_0^* K \cong \bC_{\Delta_G \times M}$, where $i_0$ is the inclusion $\{0\} \into \bR$, 
	\item the projection $\dot{\msp}(K) \to \ol{\pi}_{M \times \bR} (\dot{\msp}(K))$ is a diffeomorphism, and
	\item we have
	\begin{equation}\label{eq:ssKid}
		\ol{\pi}_{M \times \bR} (\dot{\msp}(K)) = \{(e^{t d \varphi_m(\xi)}g, g, \xi, \sm \xi, m, t) \}_{(g, \xi, m, t) \in G \times \dot{\fg}^* \times M \times \bR}, 
	\end{equation}
	where we use the right-invariant trivialization $\dot{T}^*(G \times G) \cong G \times G \times \dot{\fg}^* \times \dot{\fg}^*$.
\end{enumerate}
The sheaf $\ol{a}^* \cP$ satisfies the first condition since
$$ \ol{i}_0^* \ol{a}^* \cP \cong \ol{a}^* \ol{i}_0^* \cP \cong \ol{a}^* \bC_{\{e\} \times M} \cong \bC_{\Delta_G \times M},$$
so we are reduced to checking its singular support satisfies the remaining conditions. 

We first compute $\dot{\msp}(\cP)$. Since $\dot{\msp}(\bC_e) = \dot{T}^*_e G$, it follows from \cite[Prop. 3.12]{GKS12} and Proposition~\ref{prop:GKSongroupsflowpart} that $\dot{\msp}(\cP)$ is determined by the conditions that
\begin{enumerate}
	\item the projection $\dot{\msp}(\cP) \to \ol{\pi}_{M \times \bR} (\dot{\msp}(\cP))$ is a diffeomorphism, and  
	\item we have
	\begin{equation}\label{eq:ssPid}
		\ol{\pi}_{M \times \bR} (\dot{\msp}(\cP)) = \{(e^{t d \varphi_m(\xi)}, \xi, m, t) \}_{(\xi, m, t) \in \dot{\fg}^* \times M \times \bR},
	\end{equation}
	where we use the right-invariant trivialization $\dot{T}^*(G) \cong G  \times \dot{\fg}^*$. 
\end{enumerate}

To determine $\dot{\msp}(\ol{a}^* \cP)$ from $\dot{\msp}(\cP)$, we must compute how covectors pullback along $a$. Recall that for any $g, h \in G$, the right-invariant trivialization identifies the differential $dm_{(g,h)}: T_{(g,h)}(G \times G) \to T_{gh} G$ of the multiplication map with the map $\fg \oplus \fg \to \fg$ given by $(v,w) \mapsto v + \Ad_g w$. Likewise, the differential $d\iota_g: T_g G \to T_{g^{-1}} G$ of the inversion map is identified with $v \mapsto \sm \Ad_{g^{-1}}$. By the chain rule, $da_{(g,h)}$ is then identified with $(v,w) \mapsto v - \Ad_{gh^{-1}} w$. It now follows that $a^*_{(g, h)}: T_{gh^{-1}}^* G \to T^*_{(g,h)} (G \times G)$ is identified with the map $\fg^* \to \fg^* \oplus \fg^*$ given by $\xi \mapsto (\xi, \sm \Ad_{gh^{-1}} \xi)$. 

Given the previous two paragraphs, and given that $a^{-1}(e^{t d \varphi_m(\xi)}) = \{(e^{t d \varphi_m(\xi)}g,g)\}_{g \in G}$, the formula of \cite[Prop. 5.4.5]{KS94} yields that
$$ \ol{\pi}_{M \times \bR} (\dot{\msp}(\ol{a}^* \cP)) = \{(e^{t d \varphi_m(\xi)} g, g, \xi, \sm \Ad^*_{e^{t d \varphi_m(\xi)}} \xi, m, t) \}_{(g, \xi, m, t) \in G \times \dot{\fg}^* \times M \times \bR},$$
and that $\dot{\msp}(\ol{a}^*\cP) \to \ol{\pi}_{M \times \bR} (\dot{\msp}(\ol{a}^*\cP))$ is a diffeomorphism. By exponentiating~(\ref{eq:vanishingidentity}), however, we see that $\Ad^*_{e^{t d \varphi_m(\xi)}} \xi = \xi$ for all $t \in \bR$. Comparing (\ref{eq:ssKid}) and (\ref{eq:ssPid}), we thus obtain $\ol{\pi}_{M \times \bR} (\dot{\msp}(\ol{a}^* \cP)) = \ol{\pi}_{M \times \bR} (\dot{\msp}(K))$, concluding that $K \cong \ol{a}^* \cP$. 
\end{proof}

Finally, we study how the GKS kernels of homogeneous bi-invariant Hamiltonians behave under nearby cycles. Recall our conventions that $\smint \subseteq \bR_+$ is an open interval whose closure in~$\bR$ contains $0$, and $\lgint \subseteq \bR$ is an open interval containing $\{0\}$ and $\smint$. We denote their inclusions by $i: \{0\} \into \lgint$ and $j: \smint \into \lgint$. Given a space $X$, the associated nearby cycles functor $\psi$ is then the composition
$$ \psi: \Sh(X \times \smint) \xrightarrow{\ol{j}_*} \Sh(X \times \lgint) \xrightarrow{\ol{i}^*} \Sh(X). $$
Here we use again the convention that e.g. $\ol{i}$ denotes the product of $i$ with an identity map which is clear from context (i.e. $\ol{i} = i \times \id_X$ above). 

\begin{Proposition}\label{prop:nearbyactions}
Let $H$, $K$, and $\cP$ be as in Proposition \ref{prop:GKSongroups}, defined relative to $M = \smint$, and let $\varphi_\ep: \dot{\fg}^* \to \bR$ be the restriction of $H$ to $\dot{T}^*_e G \times \{\ep\}$ for $\ep \in \smint$. Suppose either that $G$ is compact, or that there exists $\ep' \in \smint$ and a compact subset $Z \subset \fg$ such that $d \varphi_\ep(\dot{\fg}^*) \subseteq Z$ for $\ep < \ep'$. Then if $\cF \in \Sh(G)$ is cohomologically constructible we have a diagram of natural isomorphisms 
\begin{equation}\label{eq:nearbyactions}
	\begin{tikzpicture}
		[baseline=(current  bounding  box.center),thick,>=\arrtip]
		\node (a) at (0,0) {$(\psi \cP) \conv_{\bR}\,  \cF$};
		\node (b) at (4,0) {$\psi( \cP \conv_{\smint \times \bR}\, \cF)$};
		\node (c) at (0,-1.5) {$(\psi K) \circ_{\bR}  \cF$};
		\node (d) at (4,-1.5) {$\psi(K \circ_{\smint \times \bR} \cF).$};
		\draw[->] (a) to node[above] {$\sim $} (b);
		\draw[->] (b) to node[below,rotate=90,pos=.4] {$\sim $} (d);
		\draw[->] (a) to node[below,rotate=90,pos=.4] {$\sim $}(c);
		\draw[->] (c) to node[above] {$\sim $} (d);
	\end{tikzpicture}
\end{equation}
The vertical isomorphisms are moreover defined without any constructibility hypotheses on $\cF$ and without any hypotheses on $G$ or $d \varphi_{\ep}$. 
\end{Proposition}

\begin{proof}
	Recall the isomorphism $K \circ_{\smint \times \bR} \cF \cong \cP \conv_{\smint \times \bR}\,  \cF$ of Proposition \ref{prop:GKSongroups}. We define the right isomorphism in (\ref{eq:nearbyactions}) from this by taking nearby cycles. On the other hand, the proof of Proposition \ref{prop:GKSongroups} constructs an isomorphism $K \cong \ol{a}^* \cP$, where $a: G \times G \to G$ is given by $a(g,h) = g h^{-1}$. We then have a composition
	\begin{equation*}
		\psi K \cong	\psi \ol{a}^* \cP \cong  \ol{i}^* \ol{a}^* \ol{j}_* \cP \cong \ol{a}^* \psi \cP
	\end{equation*}
	of isomorphisms, the first coming from Proposition \ref{prop:GKSongroups}, the second from smooth base change, and the last being trivial. We now define the left isomorphism in (\ref{eq:nearbyactions}) by applying Proposition~\ref{prop:groupconvidentity} to this composition. 
	
	We now define the top isomorphism in (\ref{eq:nearbyactions}), and let the bottom isomorphism then be the unique one that makes the diagram commute. 
	For convenience, we reorder the product $G \times G \times \smint \times \bR$ from Proposition \ref{prop:GKSongroups} as $\smint \times \bR \times G \times G$. In particular, this lets us rewrite the sheaf $\ol{\pi}_1^* \cP \otimes \ol{\pi}_2^* \pi_G^* \cF$ used to define $\cP \conv_{\smint \times \bR}\,\cF$ as $\cP \boxtimes \cF$. The desired isomorphism is then obtained as the following composition. 
\begin{align*}
	(\psi \cP) \conv_{\bR}\, \cF & \cong \ol{m}_! ((\ol{i}^* \ol{j}_* \cP) \boxtimes \cF)\\
	& \cong \ol{m}_! \ol{i}^* \ol{j}_* (\cP \boxtimes \cF)\\
	& \cong \ol{i}^* \ol{m}_! \ol{j}_* (\cP \boxtimes \cF)\\
	& \cong \ol{i}^*  \ol{j}_* \ol{m}_! (\cP \boxtimes \cF)\\
	& \cong \psi( \cP \conv_{\smint \times \bR}\, \cF). 
\end{align*}
Here the first and last isomorphisms are trivial, the second uses Proposition \ref{prop:boxpush}, and the third uses proper base change. The fourth comes from the base change transformation $\ol{m}_! \ol{j}_* \to  \ol{j}_* \ol{m}_!$. To show its value on $\cP \boxtimes \cG$ is indeed an isomorphism, it suffices to show that $\ol{m}$ is proper on the support of $\ol{j}_* (\cP \boxtimes \cG)$, hence that $\ol{m}_! \ol{j}_*(\cP \boxtimes \cF) \cong \ol{m}_* \ol{j}_*(\cP \boxtimes \cF)$ and $\ol{j}_* \ol{m}_!(\cP \boxtimes \cF) \cong \ol{j}_* \ol{m}_*(\cP \boxtimes \cF)$. This is trivial if $G$ is compact, so it remains to consider the case where we have a compact subset $Z \subset \fg$ as in the statement. 

We first claim it suffices to show $\supp (\cP) \subseteq \smint \times Z_{\exp}$, where $$Z_{\exp} = \{(t, e^{tz})\}_{(t,z) \in \bR \times Z} \subseteq \bR \times G.$$ 
This is equivalent to $\cP$ being a pushforward from $\smint \times Z_{\exp}$, in which case $\ol{j}_* \cP$ is a pushforward from $\lgint \times Z_{\exp}$ and $\supp (\ol{j}_* \cP) \subset \lgint \times Z_{\exp}.$ Noting again that $\ol{j}_*(\cP \boxtimes \cF) \cong (\ol{j}_* \cP) \boxtimes \cF$ by Proposition \ref{prop:boxpush}, it then follows that 
$$\supp (\ol{j}_* (\cP \boxtimes \cF)) \subseteq \lgint \times Z_{\exp} \times G.$$ 
But for any $(\ep, t, g) \in \lgint \times \bR \times G$ we have
$$ \lgint \times \bR \times G  \cap  \ol{m}^{-1}(\ep, t, g) = \{ (\ep, t, e^{tz}, e^{-tz}g) \}_{z \in Z}, $$
which is compact since it is the image of $Z$ under a continuous map. Thus $\ol{m}$ is proper on $\lgint \times \bR \times G$ and hence on $\supp (\ol{j}_* (\cP \boxtimes \cF))$. 

To show $\supp (\cP) \subseteq \smint \times Z_{\exp}$, it suffices to show that $\cP_{(\ep, t, g)} \cong 0$ for any $(\ep, t, g) \notin \smint \times Z_{\exp}$. We may assume $d\varphi_\ep$ factors through $Z$ for all $\ep \in \smint$ (otherwise replace $\smint$ with $(0, \ep')$). We may also assume $Z$ is convex and contains $0$, so that $\{e\} \subseteq tZ \subseteq t'Z$ for all $t \leq t'$ and hence $g \neq e$ (otherwise replace~$Z$ with the convex hull of $Z \cup \{0\}$). By construction, the restriction of $\cP$ to $\smint \times \{0\} \times G$ is the constant sheaf on $\smint \times \{e\}$, hence $\cP_{(\ep, 0, g)} \cong 0$. But consider the path $[0,1] \to \smint \times \bR \times G$ given by $s \mapsto (\ep,st,g)$. By (\ref{eq:ssPid}) and the assumption that $d \varphi_\ep$ factors through $Z$ for all $\ep \in \smint$, this path is in the complement of the front projection of $\dot{\msp}(\cP)$, hence $\cP_{(\ep,t,g)} \cong 0$ as well. 
\end{proof}

\section{Main results}\label{sec:mainproofs}

In this section we prove our main result, Theorem \ref{thm:mainthmintro}, treating the complete and noncomplete cases separately as Theorems \ref{thm:nonequivmainthm} and \ref{thm:noncompletemainthm}. The proofs use the equivariant version of the statement, which is proved first as Theorem \ref{thm:equivmainthm}. 

We fix throughout the following data:
\begin{enumerate}
	\item a complete fan $\Sigma \subset N_\bR$ and associated complex toric variety $X_\Sigma$, 
	\item a toric Cartier divisor~$D = \sum_{\rho \in \Sigma(1)} a_\rho D_\rho$ with support function $\varphi = \varphi_D: N_\bR \to \bR$,  
	\item an isomorphism $N_\bR \cong \bR^n$, hence an isomorphism $N_\bR \cong M_\bR$ induced by the Euclidean norm on $\bR^n$, and 
	\item a smooth non-negative $\eta: N_\bR \to \bR$ supported on $B_1(0)$ and having $\int_{N_\bR} \eta(\xi) d\xi = 1$. 
	\end{enumerate}
At the end of the section we will drop the hypothesis that $\Sigma$ is complete. 	
	
Given $\ep > 0$, we associate Hamiltonians $H_\ep: \dot{T}^* T^n \to \bR$ and $\wt{H}_\ep: \dot{T}^* M_\bR \to \bR$ to this data as follows. First, we define a homogenized smoothing $\varphi_\ep: \dot{N}_\bR \to \bR$ of $\varphi$ by 
\begin{equation*}
	\varphi_\ep(\xi) := \| \xi \| (\eta_\epsilon \conv \varphi)(\wh{\xi}),
\end{equation*}
where $\wh{\xi} := \xi/ \| \xi \|$ and where $\eta_\ep$ is given by $\eta_\ep(\xi) = \ep^{-n} \eta(\xi/\ep)$. The function $\varphi_\ep$ is essentially the conical twisting Hamiltonian of \cite[Sec. 3.1]{HH22}. More precisely, the latter is a smooth function on all of $N_\bR$ which agrees with $\varphi_\ep$ outside a large enough ball around the origin. We now take $H_\ep$ and $\wt{H}_\ep$ to be the pullbacks of $\varphi_\ep$ along the projections $\dot{T}^* T^n \cong T^n \times \dot{N}_\bR \to \dot{N}_\bR$ and $\dot{T}^* M_\bR \cong M_\bR \times \dot{N}_\bR \to \dot{N}_\bR$. 

Fixing an open interval $\smint \subseteq \bR_+$ containing $0$ in its closure, we let $H_\smint: \dot{T}^* T^n \times \smint \to \bR$ and $\wt{H}_\smint: \dot{T}^* M_\bR \times \smint \to \bR$ denote the functions whose restrictions to $\ep \in \smint$ are $H_\ep$ and $\wt{H}_\ep$. Our main objects of study are then the GKS kernels
\begin{equation*}
	K_{D,\smint} \in \Sh(T^n \times T^n \times \smint), \quad \wt{K}_{D,\smint} \in \Sh(M_\bR \times M_\bR \times \smint)
\end{equation*}
associated to the time-one flows of these families of Hamiltonians. 

To study these kernels, we will need the following result on the behavior of the derivatives of $\varphi_\ep$ as $\ep \to 0$. We note that the calculations in its proof, as well as that of Proposition~\ref{prop:limsup}, have some overlap with the calculations in the proofs of \cite[Lem. 3.1, Lem. 3.3]{HH22}. However, ultimately our needs are somewhat different, and these latter results are not themselves enough for our purposes. Below we write $\Conv(X)$ for the convex hull of a subset $X \subset M_\bR$, and for any $\tau \in \Sigma$ we write $\Sigma(n, \tau) \subset \Sigma(n)$ for the set of maximal cones having $\tau$ as a face. Recall that the relative interior $\Relint(\sigma)$ of a cone $\sigma$ is the complement of its facets. 

\begin{Proposition}\label{prop:limdphiepxi}
The divisor $D$ has $\bR$-Cartier data $\{\chi_\sigma\}_{\sigma \in \Sigma}$ such that	
\begin{equation}\label{eq:limderiv} \lim_{\ep\to 0}d \varphi_{ \ep}(\xi)=\chi_\sigma
\end{equation}
for all $\sigma \in \Sigma$ and all $\xi \in \Relint(\sigma).$ This data satisfies $\chi_\tau \in \Conv(\{\chi_\sigma\}_{\sigma \in \Sigma(n,\tau)})$ for all $\tau \in \Sigma.$  
\end{Proposition}
\begin{proof}
	We fix throughout a bounded neighborhood $B_0 \subset N_\bR$ of $0$ containing the support of $\eta$. Given a cone $\tau \in \Sigma$ we write $\Sigma(n, \tau) \subset \Sigma(n)$ for the set of maximal cones having $\tau$ as a face. We define $$\chi_\tau := \sum_{\sigma \in \Sigma(n, \tau)} \lambda_\sigma \chi_\sigma$$ where the $\lambda_\sigma$ are defined as follows. We may choose $j, k \in \bZ$ with $1 \leq j \leq k$ and $m_1, \dotsc, m_{k} \in M$ such that 
	\begin{gather*} \sigma = \{\xi \in N_\bR\, |\, \langle m_i | \xi \rangle \geq 0\text{ for }1 \leq i \leq k\},\\\tau = \{\xi \in \sigma\, |\, \langle m_i | \xi \rangle = 0\text{ for }1 \leq i \leq j\}. \end{gather*}
	We then set $\lambda_\sigma := \int_{U_\sigma} \eta(u) du,$ where
	$$ U_{\sigma} := \{u \in B_0\, |\, \langle m_i | u \rangle < 0 \text{ for } 1 \leq i \leq j\}.$$
	
	Now let $\xi$ be any point in the relative interior of $\tau$. For $\sigma \in \Sigma(n,\tau)$ and $\ep > 0$ we set 
	\begin{equation}\label{eq:Usigep}
		U_{\sigma,\ep} = \{u \in B_0\, |\, \xi - \ep u \in \Int(\sigma)\}.
		\end{equation}
	Choose $\ep_\sigma$ small enough that $\langle m_i | \xi - \ep_\sigma u \rangle > 0$ for all $i > j$ and all $u \in B_0$. Since $\langle m_i | \xi- \ep u \rangle = -\ep \langle m_i | u \rangle$ for $1 \leq i \leq j$, we further see that $ U_{\sigma,\ep} = U_\sigma$ for $\ep < \ep_\sigma$. Taking $\ep_0$ to be the minimum of the $\ep_\sigma$, it follows that $ U_{\sigma,\ep} = U_\sigma$ for all $\ep < \ep_0$ and all $\sigma \in \Sigma(n,\tau)$. For such $\ep$ the $U_{\sigma,\ep}$ are moreover disjoint and cover a dense subset of $B_0$, hence
	$$ \sum_{ \sigma \in \Sigma(n,\tau)} \lambda_\sigma = \sum_{ \sigma \in \Sigma(n,\tau)} \int_{U_\sigma} \eta(u) du = \int_{B_0} \eta(u) du = 1.$$
	Thus $\chi_\tau$ is a convex combination of the $\chi_\sigma$ for $\sigma \in \Sigma(n,\tau)$ as claimed. That the $\chi_\tau$ collectively form an $\bR$-Cartier datum follows since each of the $\chi_\sigma$ with $\sigma \in \Sigma(n,\tau)$ take the same value on any generator of a ray of $\tau$, hence $\chi_\tau$ too takes this value. 
	
	Recalling that $f_\ep := \eta_\ep \conv \varphi$, it further follows that for $\ep < \ep_0$ we have
		\begin{align*}
		f_\ep(\xi) = \int_{N_\bR} \eta_\ep(u) \varphi(\xi - u) du 
		= \int_{B_0} \eta(u) \varphi(\xi - \ep u) du 
		= \sum_{ \sigma \in \Sigma(n,\tau)} \int_{U_\sigma} \eta(u) \varphi(\xi - \ep u) du.
	\end{align*}
	Since $\varphi$ is smooth and $d \varphi$ is identically equal to $\chi_{\sigma}$ on $\Int(\sigma)$, we then further have
	\begin{align}\label{eq:dfep}
		d f_\ep(\xi) =  \sum_{ \sigma \in \Sigma(n,\tau)} \int_{U_\sigma} \eta(u) d \varphi(\xi - \ep u) du = \sum_{ \sigma \in \Sigma(n, \tau)} \lambda_\sigma \chi_{\sigma} = \chi_\tau.
	\end{align}
		An elementary calculation using the definition $\varphi_\ep(\xi) := \| \xi \| f_\ep(\wh{\xi})$ and our identification $M_\bR\cong N_\bR$ now yields 
	\begin{equation}\label{homogeneousgradient}
		d \varphi_{\ep}(\xi) = d f_\ep(\wh{\xi}) + (f_\ep(\wh{\xi}) - \langle d f_\ep(\wh{\xi}) | \wh{\xi}\rangle ) \wh{\xi}.
	\end{equation}
	Note that $\wh{\xi}$ belongs to the relative interior of $\tau$ since $\xi$ does, hence $d f_\ep(\wh{\xi}) = \chi_\tau$ since $\xi$ was chosen arbitrarily in the above analysis. Since $\varphi$ is continuous $\lim_{\ep \to 0}f_\ep(\wh{\xi}) = \varphi(\wh{\xi})$, and it follows from (\ref{homogeneousgradient}) and the previous paragraph that
	\begin{align}\label{eq:derivlim}
		\lim_{\ep \to 0}d \varphi_{ \ep}(\xi)= \chi_\tau + \left( \varphi(\wh{\xi}) - \langle \chi_\tau | \wh{\xi}\rangle \right) \wh{\xi}. 
	\end{align}	
	If $\tau$ is a face of $\sigma$, then $\varphi(\wh{\xi}) = \langle \chi_{\sigma} | \wh{\xi} \rangle$. But then since $\sum_{ \sigma \in \Sigma(n)} \lambda_\sigma = 1$ we have $\varphi(\wh{\xi}) = \langle \chi_\tau | \wh{\xi} \rangle$, so the parenthetical term in (\ref{eq:derivlim}) vanishes and the desired identity (\ref{eq:limderiv}) follows. 
\end{proof}

The preceding result describes the pointwise limit of the $d\varphi_\ep$ as $\ep \to 0$. The fact that the~$d\varphi_\ep$ are continuous but their pointwise limit is not implies this convergence is not uniform. Instead, for a typical cone $\sigma$ there will be some $R$ such that no matter how small $\ep$ is, we can find $\xi \in \Relint(\sigma)$ with $d\varphi_\ep(\xi) \notin B_R(\chi_\sigma).$ However, the following two results give us some control over this failure of uniform convergence. 

The first result quantifies how the directions along which $d\varphi_\ep(\xi) - \chi_\sigma$ can lie for small $\ep$ are constrained by the continuity of $\varphi$. Here we recall that given a space $X$ and a family $A_\ep \subset X$ of subspaces parametrized by $\ep \in \bR_+$, one defines
$$ \limsup_{\ep \to 0} A_\ep = \{x \in X\,|\, \liminf_{\ep \to 0} d(x,A_\ep) = 0\}.$$
Equivalently, $ \limsup_{\ep \to 0} A_\ep $ is obtained by taking the closure of 
$$ \{(x, \ep) | x \in A_\ep\} \subseteq X \times \bR_+ $$
in $X \times \bR_{\geq 0}$ and then intersecting with $X \times \{0\}$. Recall also that $\tau \preceq \sigma$ indicates that $\tau$ is a facet of $\sigma$. 

\begin{Proposition}\label{prop:limsup}
In $M_\bR \times \dot{N}_\bR$ we have a containment
\begin{equation}\label{eq:limsupinc}
	\limsup_{\ep \to 0}\, \{(d\varphi_\ep(\xi),\sm\xi)\}_{\xi \in \dot{\sigma}} \subseteq (\chi_\sigma, 0) + \bigcup_{\tau \preceq \sigma} \tau^\perp \times \sm\dot{\tau}
\end{equation}
for any cone $\sigma \in \Sigma$. 
\end{Proposition}
\begin{proof}
Recall from (\ref{homogeneousgradient}) that $d \varphi_\ep(\xi) = df_\ep(\wh{\xi}) + g_\ep(\xi) \wh{\xi}$ for all $\xi \in \dot{N}_\bR$, where $f_\ep := \eta_\ep \conv \varphi$ as before, and where we now define $g_\ep: \dot{N}_\bR \to \bR$ by $g_\ep(\xi) = f_\ep(\wh{\xi}) - \langle d f_\ep(\wh{\xi}) | \wh{\xi}\rangle$. We first claim the~$g_\ep$ converge uniformly to zero as $\ep \to 0$. To see this, let $\sigma_1, \dotsc, \sigma_k$ be the maximal cones of $\Sigma$ and set
$$U_{i, \ep, \xi} = \{u \in B_1(0)\, |\, \xi - \ep u \in \Int(\sigma_i)\}$$
for any $\ep > 0$, $\xi \in \dot{N}_\bR$, and  $1 \leq i \leq k$. We then have 
\begin{equation}\label{eq:uniformboundgep}
\begin{aligned}
	\left| g_\ep(\xi) \right| &\leq \int_{B_1(0)} \eta(u) \left| \varphi(\wh{\xi} - \ep u) - \langle d \varphi(\wh{\xi} - \ep u) | \wh{\xi}  \rangle \right| du \\
	&= \sum_{i = 1}^k \int_{U_{i, \ep, \wh{\xi}}} \eta(u) \left| \langle \chi_{\sigma_i} | \wh{\xi} - \ep u \rangle - \langle \chi_{\sigma_i} | \wh{\xi}  \rangle \right| du  \\
	&= \sum_{i = 1}^k \int_{U_{i, \ep, \wh{\xi}}} \eta(u)  \left|\langle \chi_{\sigma_i} | \shortminus \ep u \rangle \right|  du  \\
	&= \ep \sum_{i = 1}^k \int_{U_{i, \ep, \wh{\xi}}} \eta(u)  \left| \langle \chi_{\sigma_i} | \shortminus u \rangle \right|  du  \\
	&\leq \ep \sum_{i = 1}^k  \left|\chi_{\sigma_i} \right|.
\end{aligned} 
\end{equation}
Since $\sum_{i = 1}^k \left|\chi_{\sigma_i} \right|$ is independent of $\xi$, it follows that $\lim_{\ep \to 0} \sup_{\xi \in \dot{N}_\bR} |g_\ep(\xi)| = 0$. 

Now suppose $\{(d\varphi_{\ep_n}(\xi_n),\sm\xi_n)\}$ is a convergent sequence in $M_\bR\times \dot{N}_\bR$ with $\ep_n \to 0$ and $\xi_n \in \dot{\sigma}$ for all $n$. To prove the proposition, it suffices to show that $\lim_{n \to \infty} (d\varphi_{\ep_n}(\xi_n),\xi_n)$ belongs to the right-hand side of (\ref{eq:limsupinc}). Since $\dot{\sigma} \subset \dot{N}_\bR$ is closed, it contains $\xi := \lim_{n \to \infty} \xi_n$. We write $\tau$ for the unique face of $\sigma$ whose relative interior contains $\xi$. It then suffices to show $\lim_{n \to \infty} d\varphi_{\ep_n}(\xi_n) \in \chi_\sigma + \tau^\perp$. 

By the first paragraph the $g_{\ep_n}$ converge uniformly to zero as $n \to \infty$, hence we have $\lim_{n \to \infty} g_{\ep_n}(\xi_n) = 0$. It follows that
$$ \lim_{n \to \infty} d\varphi_{\ep_n}(\xi_n) 
= \lim_{n \to \infty} df_{\ep_n}(\wh{\xi}_n) + g_{\ep_n}(\xi_n) \wh{\xi}_n
= \lim_{n \to \infty} df_{\ep_n}(\wh{\xi}_n),$$
so it further suffices to show $\lim_{n \to \infty} df_{\ep_n}(\wh{\xi}_n) \in \chi_\sigma + \tau^\perp$.

To see this, choose $\ep > 0$ so that $B_{\ep}(\wh{\xi})$ only meets maximal cones having $\tau$ as a face. We may further choose $N$ so that $d(\wh{\xi}, \wh{\xi}_n) < \frac{\ep}{2}$ and $\ep_n < \frac{\ep}{2}$ for $n > N$. It follows that $B_{\ep_n}(\wh{\xi}_n)$ only meets maximal cones having $\tau$ as a face for $n > N$. We set $a_{i,n} = \int_{U_{i, \ep_n, \wh{\xi}_n}} \eta(u) du$, noting that for all $n$ we have $\sum_{i=1}^k a_{i,n} = 1$, and that for $n > N$ we have $a_{i,n} = 0$ unless $\tau$ is a face of $\sigma_i$. We then have as in (\ref{eq:dfep}) that
\begin{equation}\label{eq:somedfepn} d f_{\ep_n}(\wh{\xi}_n) =  \sum_{ i = 1}^k \int_{U_{i, \ep_n, \wh{\xi}_n}} \eta(u) d \varphi(\wh{\xi}_n - {\ep_n} u) du = \sum_{ i = 1}^k a_{i,n} \chi_{\sigma_i} = \chi_\sigma + \sum_{ i = 1}^k a_{i,n} (\chi_{\sigma_i} - \chi_\sigma).  
	\end{equation}
Recall that $\chi_\sigma$ and $\chi_{\sigma_1}, \dotsc,\chi_{\sigma_k}$ are part of an $\bR$-Cartier datum by Proposition \ref{prop:limdphiepxi}. If $\tau$ is a face of~$\sigma_i$, it follows that $\chi_{\sigma_i} - \chi_\sigma \in \tau^\perp$. But then by (\ref{eq:somedfepn}) we have $df_{\ep_n}(\wh{\xi}_n) \in \chi_\sigma + \tau^\perp$ for $n > N$, hence  $\lim_{n \to \infty} df_{\ep_n}(\wh{\xi}_n) \in \chi_\sigma + \tau^\perp$. 
\end{proof}

Next, we show that as $\ep \to 0$ the values of $d \varphi_\ep$ are uniformly bounded near the convex hull of the $\chi_\sigma$. Again we write $\Conv(X)$ for the convex hull of a subset $X \subset M_\bR$, and $B_R(X)$ for the set of points whose distance to $X$ is at most $R$. Note also that the convexity claim of Proposition \ref{prop:limdphiepxi} implies that $\Conv(\{\chi_\sigma\}_{\sigma \in \Sigma})$ could equivalently be written as $\Conv(\{\chi_\sigma\}_{\sigma \in \Sigma(n)})$. 

\begin{Proposition}\label{prop:dvarphiepbounded}
If $R = \sum_{\sigma \in \Sigma(n)} \left|\chi_{\sigma} \right|$, then for any $\ep > 0$ we have 
$$ d \varphi_\ep(\dot{N}_\bR) \subseteq \overline{B_{\ep R}(\Conv(\{\chi_\sigma\}_{\sigma \in \Sigma}))}. $$
\end{Proposition}
\begin{proof}
Most of the necessary work has already been done. Recall again from (\ref{homogeneousgradient}) that for all $\xi \in \dot{N}_\bR$ and $\ep > 0$ we have $d \varphi_\ep(\xi) = df_\ep(\wh{\xi}) + g_\ep(\xi) \wh{\xi}$, where $f_\ep := \eta_\ep \conv \varphi$ and where $g_\ep(\xi) = f_\ep(\wh{\xi}) - \langle d f_\ep(\wh{\xi}) | \wh{\xi}\rangle$. It further follows from (\ref{eq:dfep}) that $df_\ep(\wh{\xi}) \in \Conv(\{\chi_\sigma\}_{\sigma \in \Sigma})$. On the other hand, we have shown in (\ref{eq:uniformboundgep}) that $\left| g_\ep(\xi) \right| \leq \ep R$. The claim now follows since by definition we have $| \wh{\xi} | = 1$. 
\end{proof}

Note that for any face $\tau \preceq \sigma$, we have $\tau^\perp \subset \rho^\perp$ for any ray $\rho \in \tau(1) \subseteq \sigma(1)$. In particular, if we set
$$ \sigma(1)^\perp := \bigcup_{\rho \in \sigma(1)} \rho^\perp, $$
then Proposition \ref{prop:limsup} implies that
$$ \limsup_{\ep \to 0} d\varphi_\ep(\dot{\sigma}) \subseteq \chi_\sigma + \sigma(1)^\perp.$$ 
In what follows we will use this in conjunction with the following result, which lets us control sheaves whose singular support lies above a translate of $\sigma(1)^\perp$ and satisfies a directionality condition. In the statement we write $j$ for the inclusion $\Int(\sigma^\vee) \into M_\bR$, noting that $\Int(\sigma^\vee)$ is a connected component of the complement of $\sigma(1)^\perp$. 

\begin{Proposition}\label{prop:sigma1perpssprop}
Suppose that $\cF \in \Sh^b(M_\bR)$ satisfies
$$\dot{\msp}(\cF) \subseteq \sigma(1)^\perp \times (\dot{N}_\bR \smallsetminus \dot{\sigma}), $$
and that $\cF_x \cong 0$ for all $x \notin \sigma(1)^\perp \cup \Int(\sigma^\vee)$. Then the natural map $j_! j^! \cF \to \cF$ is an isomorphism. 
\end{Proposition}
\begin{proof}
The claim is equivalent to the claim that $\cF_x \cong 0$ for all $x \notin \Int(\sigma^\vee)$. We prove this by induction on 
$$ r(x) = \# \{ \rho \in \sigma(1) \,|\, x \in \rho^\perp \}.$$
If $r(x)=0$ then $x \notin \Int(\sigma^\vee)\cup \sigma(1)^\perp$, hence by hypothesis $\cF_x \cong 0$. Assume then that $r(x) > 0$ and that $\cF_y \cong 0$ for all $y \notin \Int(\sigma^\vee)$ with $r(y) < r(x)$. 

Choose any $\rho \in \sigma(1)$ such that $x \in \rho^\perp$, and let $\xi \in \rho$ be a generator. We have inclusions
$$ i_\rho: Z_\rho^+ = \{ y \,|\, \langle \xi \,|\, y \rangle \geq 0 \} \into M_\bR, \quad j_\rho: U_\rho^- = \{ y \,|\, \langle \xi \,|\, y \rangle < 0 \} \into M_\bR. $$
Since the differential of $\langle \xi \,|\, - \rangle$ at $x$ is $(x, \xi)$, we must have $(i_{\rho!}i_\rho^! \cF)_x \cong 0$. Otherwise $(x, \xi) \in \dot{\msp}(\cF)$, which contradicts our hypotheses since $\xi \in \sigma$. Thus the natural map $\cF_x \to (j_{\rho *} j_\rho^* \cF)_x$ is an isomorphism. But $U_\rho^-$ lies in the complement of $\Int(\sigma^\vee)$, and we can choose a neighborhood $U_x$ of $x$ such that $r(y) < r(x)$ for all $y \in U_x \cap U_\rho^-$. By our inductive assumption $j_\rho^* \cF$ must then vanish on $U_x \cap U_\rho^-$. It follows that $(j_{\rho *} j_\rho^* \cF)_x$ is zero, hence so is~$\cF_x$. 
\end{proof}

Recall that the twisted polytope sheaf $\cP(D)$ admits a Cech resolution whose summands are of the form $\coeff_{\,\Int(\chi_\sigma + \sigma^\vee)}$ for $\sigma \in \Sigma$. To relate the GKS kernel $\wt{K}_{D,\smint}$ to $\cP(D)$, we first relate it to these summands. 

\begin{Proposition}\label{prop:psiKDonshards}
	There is an isomorphism
	$$ \psi(\wt{K}_{D, \smint} \circ_\smint \coeff_{\,\Int(\sigma^\vee)}) \cong \coeff_{\,\Int(\chi_\sigma + \sigma^\vee)}$$
	for any $\sigma \in \Sigma$. 
\end{Proposition}
\begin{proof}
Recall that $\wt{K}_{D,\smint}$ is obtained from a larger family of kernels $\wt{K}_{D,\smint,\bR} \in \Sh(M_\bR \times M_\bR \times \smint \times \bR)$ by restricting to $1\in\bR$. Let us similarly write $\wt{K}_{D,\ep,t} \in \Sh(M_\bR \times M_\bR)$ for the result of restricting this family to $(\ep,t)\in \smint \times \bR$. We begin by establishing some bounds on $\dot{\msp}(\wt{K}_{D,\ep,t} \circ \coeff_{\,\Int(\sigma^\vee)})$. 

Note that the right-hand side of (\ref{eq:limsupinc}) is contained in $(\chi_\sigma,0) + \sigma(1)^\perp \times \sm\dot{\sigma}$, similarly if we replace $\sigma$ by any of its faces. As there are finitely many faces, it follows from Proposition \ref{prop:limsup} that we may choose $\delta,\ep' > 0$ so that for $\ep < \ep'$ we have 
\begin{equation}\label{eq:dvphicontaineq} (d\varphi_\ep(\xi),\sm\xi) \in B_\delta((\chi_\tau,0) + \tau(1)^\perp \times \sm\dot{\tau}) \end{equation}
for all faces $\tau$ and all $\xi \in \dot{\tau}$. We then claim that
\begin{equation}\label{eq:ssshardcontainment}
\begin{aligned}
\dot{\msp}(\wt{K}_{D,\ep,t} \circ \coeff_{\,\Int(\sigma^\vee)}) & \subseteq \bigcup_{\tau \preceq \sigma} \bigcup_{\xi \in \dot{\tau}} \,(t d\varphi_{\ep}(\xi),\sm\xi) + (\tau^\perp \cap \sigma^\vee) \times \{0\} \\
& \subseteq \bigcup_{\tau \preceq \sigma} B_{t \delta}((t \chi_\tau,0) + \tau(1)^\perp \times \sm\dot{\tau}) + (\tau^\perp \cap \sigma^\vee) \times \{0\} \\
& \subseteq \bigcup_{\tau \preceq \sigma} B_{t \delta}((t \chi_\tau,0) + \tau(1)^\perp \times \sm\dot{\tau}) \\
& \subseteq B_{t \delta}((t \chi_\sigma,0) + \sigma(1)^\perp \times \sm\dot{\sigma}). 
\end{aligned}
\end{equation}
The first containment follows from the fact that $\dot{\msp}(\coeff_{\,\Int( \sigma^\vee)}) = \bigcup_{\tau \preceq \sigma} (\tau^\perp \cap \sigma^\vee) \times (\sm \dot{\tau})$, the second from (\ref{eq:dvphicontaineq}), the third from the fact that $\tau(1)^\perp$ is closed under addition by $\tau^\perp$, and the last from the fact that $\chi_\tau + \tau(1)^\perp \subset \chi_\sigma + \sigma(1)^\perp$ for all $\tau \preceq \sigma$. 

Since $\delta$ can be arbitrarily small in the previous paragraph, it follows that 
$$ \limsup_{\ep \to 0} \dot{\msp}(\wt{K}_{D,\ep,1} \circ \coeff_{\,\Int(\sigma^\vee)}) \subseteq (\chi_\sigma,0) + \sigma(1)^\perp \times \sm\dot{\sigma}. $$ 
It then further follows from the singular support bound of \cite[Lem. 3.16]{NS20} that
\begin{equation}\label{eq:vcyclesssbound} \dot{\msp}(\psi (\wt{K}_{D, \smint} \circ_\smint \coeff_{\, \Int(\sigma^\vee)})) \subseteq (\chi_\sigma,0) + \sigma(1)^\perp \times \sm \dot{\sigma}. \end{equation}
Aiming to apply Proposition \ref{prop:sigma1perpssprop}, we now compute the stalks of $\psi (\wt{K}_{D, \smint} \circ_\smint \coeff_{\,\Int(\sigma^\vee)})$ in the complement of $\chi_\sigma + \sigma(1)^\perp$. 

To do this, let $V$ denote the complement of the closure of $\bigcup_{0 \leq t \leq 1} B_\delta(t\chi_\sigma + \sigma(1)^\perp)$. Each component of $V$ is contained in a unique component of the complement of $\sigma(1)^\perp$, and also in a unique component of the complement of $\chi_\sigma + \sigma(1)^\perp$. For any $x \in V$, it follows from (\ref{eq:ssshardcontainment}) and \cite[Prop. 3.2(i)]{GKS12} that $(\coeff_{\,\Int(\sigma^\vee)})_x \cong (\wt{K}_{D,\ep,1} \circ \coeff_{\,\Int(\sigma^\vee)})_x$. If $x \in V$ is not contained in $\Int(\sigma^\vee)$, it immediately follows that $(\wt{K}_{D,\ep,1} \circ \coeff_{\,\Int(\sigma^\vee)})_x \cong 0.$ This being true in a neighborhood of $x$ and for all $\ep < \ep'$, it follows that $\psi(\wt{K}_{D, \smint} \circ_\smint \coeff_{\,\Int(\sigma^\vee)})_x \cong 0$. Since each component of the complement of $\chi_\sigma + \sigma(1)^\perp$ other than $\chi_\sigma + \Int(\sigma^\vee)$ contains a point of this form, it follows from (\ref{eq:vcyclesssbound}) that $\psi(\wt{K}_{D, \smint} \circ_\smint \coeff_{\,\Int(\sigma^\vee)})$ is zero on these components. 

It now follows from Proposition \ref{prop:sigma1perpssprop} that the map 
$$j_! j^! \psi(\wt{K}_{D, \smint} \circ_\smint \coeff_{\,\Int(\sigma^\vee)}) \to \psi(\wt{K}_{D, \smint} \circ_\smint \coeff_{\,\Int( \sigma^\vee)})$$ 
is an isomorphism, where $j$ is the inclusion $\Int(\sigma^\vee) \into M_\bR$. Choosing $x \in V \cap \Int(\sigma^\vee)$, it follows from the argument of the previous paragraph that $\psi(\wt{K}_{D, \smint} \circ_\smint \coeff_{\,\Int( \sigma^\vee)})_x \cong \bC$. But (\ref{eq:vcyclesssbound}) also implies $\dot{\msp}(\psi(\wt{K}_{D, \smint} \circ_\smint \coeff_{\,\Int( \sigma^\vee)}))$ lives above the complement of $\Int(\sigma^\vee)$, hence $j^! \psi(\wt{K}_{D, \smint} \circ_\smint  \coeff_{\,\Int(\sigma^\vee)})$ is locally constant and the claim follows. 
\end{proof}

We are now ready to prove our main results, first in the equivariant setting and then in the nonequivariant setting. As reviewed in Section \ref{sec:CCC}, we denote the functors comprising the equivariant and nonequivariant CCCs respectively by 
$$ \wt{A}: \Perf^{T_N}(X_\Sigma) \rightleftarrows \Sh_{\wt{\Lambda}_\Sigma}^{b,cpt}(M_\bR) :\wt{B}$$ 
and
$$ A: \Coh(X_\Sigma) \rightleftarrows \Sh_{\Lambda_\Sigma}^c(T^n) :B.  $$
We refer to  Remark \ref{rem:compactvsconstruct} below for some discussion of the distinction between compactness and constructibility in the following statements. 

\begin{Theorem}\label{thm:equivmainthm}
	For any $\cF \in \Sh(M_\bR)$ we have an isomorphism
	$$ \psi(\wt{K}_{D,\smint}) \circ \cF \cong \cP(D) \conv \cF. $$
In particular, if $\cF \in \Sh_{\wt{\Lambda}_\Sigma}^{b,cpt}(M_\bR)$ we have an isomorphism
	$$ \psi(\wt{K}_{D,\smint}) \circ \cF \cong \wt{A}(\cO(D) \otimes \wt{B}(\cF)). $$
\end{Theorem}
\begin{proof}
We first claim that $\psi(\wt{K}_{D,\smint} \circ_\smint \bC_0) \cong \cP(D).$ Recall that $\cP(D)$ is defined by the resolution
$$   \cdots \to \bigoplus_{\sigma \in \Sigma(k+1)} \bC_{\Int(\chi_\sigma + \sigma^\vee)}[k+1] \to \bigoplus_{\sigma \in \Sigma(k)} \bC_{\Int(\chi_\sigma + \sigma^\vee)}[k] \to \cdots,  $$
where the differentials are the sums of the canonical maps $\bC_{\Int(\chi_\sigma + \sigma^\vee)} \to \bC_{\Int(\chi_\tau + \tau^\vee)}$ for $\tau$ a facet of $\sigma$. As $\bC_0$ is the twisted polytope sheaf of the trivial divisor, it likewise admits the resolution
$$   \cdots \to \bigoplus_{\sigma \in \Sigma(k+1)} \bC_{\Int( \sigma^\vee)}[k+1] \to \bigoplus_{\sigma \in \Sigma(k)} \bC_{\Int(\sigma^\vee)}[k] \to \cdots.  $$
By Proposition \ref{prop:psiKDonshards} we have $\psi(\wt{K}_{D,\smint} \circ_\smint \bC_{\Int( \sigma^\vee)}) \cong  \bC_{\Int(\chi_\sigma + \sigma^\vee)}$ for all $\sigma \in \Sigma$. It follows from the proof that the value of the functor $\psi(\wt{K}_{D,\smint} \circ_\smint -)$ on the canonical map $\bC_{\Int(\sigma^\vee)} \to \bC_{\Int( \tau^\vee)}$ is the canonical map $\bC_{\Int(\chi_\sigma + \sigma^\vee)} \to \bC_{\Int(\chi_\tau + \tau^\vee)}$, since both have the same value on the stalk of a suitably chosen point in $\Int(\chi_\sigma + \sigma^\vee)$. Thus we obtain $\psi(\wt{K}_{D,\smint} \circ_\smint \bC_0) \cong \cP(D)$.

By Proposition \ref{prop:nearbyactions}, or more precisely its restriction to $1 \in \bR$, we now have
$$ \psi(\wt{K}_{D,\smint}) \circ \cF \cong \cP(D) \conv \cF$$
for any $\cF \in \Sh(M_\bR)$. But for $\cF \in \Sh_{\wt{\Lambda}_\Sigma}^{b,cpt}(M_\bR)$ we have $\cP(D) \conv \cF \cong A(\cO(D) \otimes B(\cF))$ by \cite[Cor. 3.5, Cor. 3.13]{FLTZ11}, and the proposition follows. 
\end{proof}

The next result establishes Theorem \ref{thm:mainthmintro} under our standing assumption that $\Sigma$ is complete. 

\begin{Theorem}\label{thm:nonequivmainthm}
	For any $\cF \in \Sh(T^n)$ we have an isomorphism
	\begin{equation*}\label{eq:nonequivmainthm} \psi(\no{K}_{D,\smint}) \circ \cF \cong (p_! \cP(D)) \conv \cF. \end{equation*}
	In particular, if $\cF \in \Sh_{\Lambda_\Sigma}^c(T^n)$ we have an isomorphism
	$$ \psi(\no{K}_{D,\smint}) \circ \cF \cong A(\cO(D) \otimes B(\cF)). $$
\end{Theorem}
\begin{proof}
We write $\ol{p}_1$ and $\ol{p}_2$ for $p \times \id_{M_\bR}$ and $\id_{T^n} \times p$, respectively, or for their product with an identity map which is clear from context. 
Let us first prove the claim assuming there is an isomorphism $\ol{p}_{1!} \psi(\wt{K}_{D,\smint}) \cong \ol{p}_2^* \psi(\no{K}_{D,\smint})$. We would then have isomorphisms
\begin{align*}
\psi(\no{K}_{D,\smint} \circ_\smint \bC_0) \cong \psi(\no{K}_{D,\smint}) \circ \bC_0 \cong p_!(\psi(\wt{K}_{D,\smint}) \circ \bC_0) \cong p_! \cP(D).
\end{align*}
Here the first isomorphism uses Proposition \ref{prop:nearbyactions}, the second uses Proposition \ref{prop:convprojform} and our assumption, and the third uses Theorem \ref{thm:equivmainthm}. The desired isomorphism between $\psi(\no{K}_{D,\smint}) \circ \cF$ and $ (p_! \cP(D)) \conv \cF$ now follows from Proposition \ref{prop:nearbyactions}. For $\cF \in \Sh_{\Lambda_\Sigma}^c(T^n)$, the isomorphism between $\psi(\no{K}_{D,\smint}) \circ \cF$ and $A(\cO(D) \otimes B(\cF))$ further follows from \cite[Rem. 12.12]{Kuw20} and~(\ref{eq:equivnonequivdiagram}), noting that $\cO(D) \otimes B(\cF) \cong (\cO(D) \otimes \omega_{X_\Sigma}) \overset{!}{\otimes} B(\cF)$. 

To show $\ol{p}_{1!} \psi(\wt{K}_{D,\smint}) \cong \ol{p}_2^* \psi(\no{K}_{D,\smint})$, we first show $\ol{p}_{1!} \wt{K}_{D,\smint} \cong \ol{p}_2^* \no{K}_{D,\smint}$. We again write $\wt{K}_{D,\smint,\bR} \in \Sh(M_\bR \times M_\bR \times \smint \times \bR)$ and $\no{K}_{D,\smint,\bR} \in \Sh(T^n \times T^n \times \smint \times \bR)$ for the full GKS kernels quantizing the flows of $\wt{H}_\smint$ and $H_\smint$ for all $t \in \bR$. It follows from (\ref{eq:ssKid}) in the proof of Proposition \ref{prop:GKSongroups} that $\dot{\msp}(\ol{p}_{1!} \wt{K}_{D,\smint,\bR})$ and $\dot{\msp}(\ol{p}_2^* \no{K}_{D,\smint,\bR})$ are both contained in the conic Lagrangian $L$ determined by 
$$  \ol{\pi}_{\smint \times \bR}(L) = \{(p(x + t d\varphi_\ep(\xi)), x, \xi, \sm \xi, \ep, t) \}_{(x, \xi, \ep, t) \in \dot{T}^*M_\bR \times \smint \times \bR}.  $$
As in the proof of \cite[Prop. 3.2(ii)]{GKS12}, it follows that $\dot{\msp}(\ol{p}_{1!} \wt{K}_{D,\smint,\bR} \circ_{\smint \times \bR} \wt{K}_{D,\smint,\bR}^{\sm 1})$ and $\dot{\msp}(\ol{p}_2^* \no{K}_{D,\smint,\bR} \circ_{\smint \times \bR} \wt{K}_{D,\smint,\bR}^{\sm 1})$ are both contained in the conormal to $$\ol{p}_1(\Delta_{M_\bR}) \times \smint \times \bR = \ol{p}_2^{\sm 1}(\Delta_{T^n}) \times \smint \times \bR.$$ 
Since the restrictions of $\ol{p}_{1!} \wt{K}_{D,\smint,\bR} \circ_{\smint \times \bR} \wt{K}_{D,\smint,\bR}^{\sm 1}$ and $\ol{p}_2^* \no{K}_{D,\smint,\bR} \circ_{\smint \times \bR} \wt{K}_{D,\smint,\bR}^{\sm 1}$ to $T^n \times M_\bR\times \smint \times \{0\}$ are both isomorphic to $\coeff_{\ol{p}_1(\Delta_{M_\bR}) \times \smint}$, it  follows that these sheaves must be isomorphic before restricting. Composing  on the right with $\wt{K}_{D,\smint,\bR}$ then yields the desired isomorphism $\ol{p}_{1!} \wt{K}_{D,\smint} \cong \ol{p}_2^* \no{K}_{D,\smint}$. 

The desired isomorphism $\ol{p}_{1!} \psi(\wt{K}_{D,\smint}) \cong \ol{p}_2^* \psi(\no{K}_{D,\smint})$ may now be obtained as the composition
\begin{equation}\label{eq:psiprojeq}
	\begin{aligned}
\ol{p}_{1!} \ol{i}^* \ol{j}_* \wt{K}_{D,\smint} & \cong \ol{i}^* \ol{p}_{1!} \ol{j}_* \wt{K}_{D,\smint} \\
 & \cong  \ol{i}^* \ol{j}_* \ol{p}_{1!} \wt{K}_{D,\smint} \\
 & \cong  \ol{i}^* \ol{j}_* \ol{p}_{2}^* \no{K}_{D,\smint} \\
 & \cong  \ol{i}^* \ol{p}_{2}^* \ol{j}_*  \no{K}_{D,\smint} \\
 & \cong  \ol{p}_{2}^* \ol{i}^* \ol{j}_*  \no{K}_{D,\smint}.
 \end{aligned}
\end{equation}
Here the first isomorphism uses proper base change, the third uses the previous paragraph, the fourth uses smooth base change, and the last is trivial. 

The second isomorphism in (\ref{eq:psiprojeq}) comes from the base change transformation $\ol{p}_{1!} \ol{j}_* \to \ol{j}_* \ol{p}_{1!}$. To show its value on $\wt{K}_{D,\smint}$ is indeed an isomorphism, it suffices to show that $\ol{p}_1$ is proper on the support of $\ol{j}_* \wt{K}_{D,\smint}$, hence that $\ol{p}_{1!} \ol{j}_* \wt{K}_{D,\smint} \cong \ol{p}_{1*} \ol{j}_* \wt{K}_{D,\smint}$ and $\ol{j}_* \ol{p}_{1!} \wt{K}_{D,\smint} \cong \ol{j}_*  \ol{p}_{1*} \wt{K}_{D,\smint}$. It follows from Proposition \ref{prop:dvarphiepbounded} that for some $R, \ep' > 0$ we have $|td \varphi_\ep( \dot{N}_\bR)| < R$ for all $\ep < \ep'$ and $t \in [0,1]$. Replacing $\smint$ with $\smint \cap (0,\ep')$ and $\lgint$ with $\lgint \cap (- \infty, \ep')$ if needed, it now follows from \cite[Prop. 3.2(i)]{GKS12} and (a trivial case of) Proposition \ref{prop:GKSongroupsflowpart} that $\supp(\wt{K}_{D,\smint})$ is contained in $B_R(\Delta_{M_\bR}) \times \smint$. Thus $\supp(\ol{j}_* \wt{K}_{D,\smint})$ is contained in $B_R(\Delta_{M_\bR}) \times \lgint$, and the claim follows since $\ol{p}_1$ is proper on $B_R(\Delta_{M_\bR}) \times \lgint$. 
\end{proof}

\begin{Remark}\label{rem:compactvsconstruct}
Note that if $\cF \in Sh(T^n)$ is cohomologically constructible, it follows from Proposition \ref{prop:nearbyactions} that the first isomorphism of Theorem \ref{thm:nonequivmainthm} is equivalent to 
$$ \psi(\no{K}_{D,\smint} \circ_\smint \cF) \cong (p_! \cP(D)) \conv \cF.$$ 
This is presumably true for more general $\cF \in Sh(T^n)$: cohomological constructibility is only needed to apply Proposition \ref{prop:boxpush} on the compatibility of pushforward with external products, and we imagine this can be relaxed somewhat. In particular, we imagine it holds when $\cF \in Sh^c_{\Lambda_\Sigma}(T^n)$, so that $\cF$ is weakly constructible but need not have perfect stalks. This is only even an issue in the singular case, since when $X_\Sigma$ is smooth $Sh^c_{\Lambda_\Sigma}(T^n) = Sh^b_{\Lambda_\Sigma}(T^n)$ and any such $\cF$ is in fact constructible. 

The same discussion holds in the equivariant setting: if $\cF \in \Sh(M_\bR)$ is cohomologically constructible, it follows from Proposition \ref{prop:nearbyactions} together with Proposition \ref{prop:dvarphiepbounded} that the first isomorphism of  \ref{thm:equivmainthm} is equivalent to 
$$ \psi(\wt{K}_{D,\smint} \circ_\smint \cF) \cong \cP(D) \conv \cF.$$ 
\end{Remark}

We now consider the case when the fan $\Sigma$ is not complete. In this case $\varphi_D$ is only defined on $|\Sigma|$, so to speak of the Hamiltonian flows of its smoothings we must first extend $\varphi_D$ to all of $N_\bR$. Ultimately this is not a serious problem, and the above proofs could be modified to work with any reasonable such extension. However, this also makes the proofs somewhat less readable, hence we have chosen to instead treat the noncomplete case separately, arguing by reduction to the complete case. 

To do this we first choose (1) a smooth refinement $\Sigma_s$ of $\Sigma$, and (2) a smooth, complete fan $\Sigma_c$ containing $\Sigma_s$. The associated Legendrians are nested in each other, i.e. we have $\Lambda_{\Sigma} \subseteq \Lambda_{\Sigma_s} \subseteq \Lambda_{\Sigma_c}$. We have an associated proper birational map $f: X_{\Sigma_s} \to X_\Sigma$ and an associated open immersion $u: X_{\Sigma_s} \to X_{\Sigma_c}$. 

\begin{Lemma}\label{lem:varphiext}
There are toric Cartier divisors $D_s$ on $X_{\Sigma_s}$ and $D_c$ on $X_{\Sigma_c}$ such that 
$$ f^* \cO(D) \cong \cO(D_s) \cong u^* \cO(D_c), $$
and such that the restriction of the support function $\varphi_{D_c}: N_\bR \to \bR$ to $|\Sigma| = |\Sigma_s|$ coincides with the support functions $\varphi_{D}$ and $\varphi_{D_s}$. 
\end{Lemma}
\begin{proof}
Note that $\varphi_D: |\Sigma| = |\Sigma_s| \to \bR$ still satisfies the conditions needed to be a support function for a Cartier divisor on $X_{\Sigma_s}$. Explicitly, for any $\rho \in \Sigma_s(1)$ with primitive generator~$u_\rho$ we set $a_\rho =  \sm \varphi_D(u_\rho)$ and
$ D_s = \sum_{\rho \in \Sigma_s(1)} a_\rho D_\rho.$
We then extend the assignment $\rho \mapsto a_\rho$ arbitrarily from $\Sigma_s(1)$ to $\Sigma_c(1)$. Since $\Sigma_c$ is smooth the resulting divisor $ D_c = \sum_{\rho \in \Sigma_c(1)} a_\rho D_\rho$ will still be Cartier. The final claim is immediate from how the correspondence between support functions and divisors is defined. 
\end{proof}

In particular, we have that $\varphi_{D_c}$ is an extension of $\varphi_D$ to all of $N_\bR$, and with this in mind we write $\varphi$ for $\varphi_{D_c}$ from now on. We then define the Hamiltonians $H_{D,\ep}$ and the family of kernels $K_{D,I}$ the same way as before, the only difference being that $\varphi$ is now our chosen extension of $\varphi_D$ rather than $\varphi_D$ itself. We then have the following final result, completing the proof of Theorem \ref{thm:mainthmintro} in the case of arbitrary $\Sigma$. 

\begin{Theorem}\label{thm:noncompletemainthm}
For any $\cF \in \Sh_{\Lambda_\Sigma}^c(T^n)$ we have an isomorphism
$$ \psi(\no{K}_{D,\smint}) \circ \cF \cong A(\cO(D) \otimes B(\cF)). $$
\end{Theorem}
\begin{proof}
Let us use the same notation for $A$ and $B$ and for their ind-extensions, likewise for the equivalences $A_s$, $B_s$ and $A_c$, $B_c$ associated to $\Sigma_s$ and $\Sigma_c$. By \cite[Lem. 14.5, Rem. 14.11]{Kuw20} these equivalences intertwine the inclusions
$$ \Sh_{\Lambda_{\Sigma}}(T^n) \subseteq   \Sh_{\Lambda_{\Sigma_s}}(T^n) \subseteq \Sh_{\Lambda_{\Sigma_c}}(T^n) $$
with the functors
$$ \IndCoh(X_\Sigma) \xrightarrow{f^!} \IndCoh(X_{\Sigma_s}) \xrightarrow{u_*} \IndCoh(X_{\Sigma_c}). $$
But we then have
\begin{align*}
A(\cO(D) \otimes B(\cF)) & \cong A_c (u_* f^!(\cO(D) \otimes B(\cF))) \\
& \cong A_c (u_*(\cO(D_s) \otimes f^!B(\cF))) \\
& \cong A_c (\cO(D_c) \otimes u_*f^!B(\cF)) \\
& \cong A_c (\cO(D_c) \otimes B_c(\cF)) \\
& \cong \psi(\no{K}_{D,\smint}) \circ \cF,
\end{align*}
where the first and fourth isomorphism use \cite[Lem. 14.5, Rem. 14.11]{Kuw20}, the second uses Lemma \ref{lem:varphiext} and the projection formula, the third uses Lemma \ref{lem:varphiext} and the compatibility of~$f^!$ with tensoring by perfect objects, and the last uses Theorem \ref{thm:nonequivmainthm}. 
\end{proof}

\begin{Remark}
We assume our main results extend from toric varieties to the toric stacks considered in \cite{Kuw20}. However, we have not attempted to think through this extension carefully, let alone document~it. It is also clear that our results apply when the ground ring is more general than $\bC$, but we have restricted our attention to this case since the results of \cite{Kuw20} are only stated in this generality. 
\end{Remark}
\bibliographystyle{amsalpha}
\bibliography{bibmirrorisotopies2}

\end{document}